\RequirePackage{fix-cm}
\documentclass[smallextended,referee,envcountsect]{svjour3}       
\smartqed  
\usepackage{graphicx}
\usepackage{amsmath, ntheorem}
\usepackage{amssymb}
\usepackage{stackrel}
\usepackage{marvosym, enumerate}
\usepackage{amstext, amscd, latexsym}
\usepackage{amsfonts, ragged2e}
\usepackage{mathrsfs, pgfplots, enumerate, setspace}
\usepackage{subfigure}
\usepackage{cases}

\smartqed
\usepackage{cite}
\usepackage{amstext, graphicx, amscd, latexsym, amssymb}
\usepackage{amsfonts, ragged2e}
\usepackage{pgfplots, setspace}
\usepackage{geometry}
\usepackage{latexsym,bm}
\usepackage{float}
 \usepackage{epstopdf,caption}

\newtheorem{theo}{Theorem}
\newtheorem{defi}{Definition}
\newtheorem{lem}{Lemma}
\newtheorem{rem}{Remark}
\newtheorem{exam}{Example}

\newtheorem{assumption}{Assumption}
\usepackage[figuresright]{rotating}
\usepackage[sort&compress,numbers]{natbib}
\usepackage{algorithm} 
\usepackage{algorithmic}  
\usepackage[algo2e]{algorithm2e} 
\usepackage[colorlinks,linkcolor=blue,anchorcolor=blue,citecolor=blue]{hyperref}

\begin{document}

\title{Variable Metric Method for Unconstrained Multiobjective Optimization Problems}

\titlerunning{Variable Metric Method for Unconstrained Multiobjective Optimization Problems}        

\author{Jian Chen$^1$ \and Gaoxi Li$^2$ \and Xinmin Yang$^3$  }


\institute{J. Chen \at Department of Mathematics, Shanghai University,
	Shanghai 200444, China\\
                    chenjian\_math@163.com\\
              G.X. Li \at School of Mathematics and Statistics, Chongqing Technology and Business University, Chongqing 400067, China\\
            ligaoxicn@126.com\\
        \Letter X.M. Yang \at National Center for Applied Mathematics of Chongqing, and School of Mathematical Sciences,  Chongqing Normal University, Chongqing 401331, China\\
        xmyang@cqnu.edu.cn  \\}

\date{Received: date / Accepted: date}

\maketitle

\begin{abstract}
In this paper, we propose a variable metric method for unconstrained multiobjective optimization problems (MOPs). First, a sequence of points is generated
using different positive definite matrices in the generic framework. It is proved that accumulation points of the sequence are Pareto critical points. Then, without convexity assumption, strong convergence is established for the proposed method. Moreover, we use a common matrix to approximate the Hessian matrices of all objective functions, along which, a new nonmonotone line search technique is proposed to achieve a local superlinear convergence rate. Finally, several numerical results demonstrate the effectiveness of the proposed method.

\keywords{Multiobjective optimization \and Variable metric method \and Pareto point \and Superlinear convergence}
\subclass{90C29 \and 90C30}
\end{abstract}

\section{Introduction}
An unconstrained multiobjective optimization problem can be stated as follows:
\begin{align*}
\min\limits_{x\in\mathbb{R}^{n}} F(x), \tag{MOP}\label{MOP}
\end{align*}
where $F:\mathbb{R}^{n}\rightarrow\mathbb{R}^{m}$ is a twice continuously differentiable function. In this framework, several objective functions have to be minimized simultaneously. Generally, there may not be a single optimal solution to reach the optima for all objectives. Instead, there is a set of compromise solutions for (\ref{MOP}), for which none of the objectives can be improved without sacrificing the other objectives. In practice, a lot of problems can be cast into MOPs. It can be found in engineering \cite{1a}, economics \cite{1b}, bioinformatics \cite{1c}, water resource management \cite{1d}, etc. As a widely used mathematical model, research has intensively been made in MOPs.
\par In the past decades, several methods have been developed to solve MOPs \cite{1,2,3,4}. A well-known strategy is the scalarization approach due to Geoffrion \cite{5}. The approach converts a MOP into a single-objective optimization problem (SOP) through a number of predefined parameters so that standard mathematical programming methods can be applied. However, the approach imposes a burden on the decision-maker to choose the parameters, which is unknown in advance. To overcome the drawback, Mukai \cite{6a} proposed the first descent method for MOPs, in which all objective functions decrease along the generated descent direction, and no prior information is needed. Fliege and Svaiter \cite{6} independently reinvented the parameter-free multiobjective optimization method, called the steepest descent method for multiobjective. Comparing with the steepest descent method for SOPs, the steepest descent direction for MOPs is given by solving the following subproblem: 
\begin{align*}
	&\min\limits_{d\in\mathbb{R}^{n}} \max\limits_{i=1,2,...,m}\ \nabla F_{i}(x)^{\mathrm T}d+\frac{1}{2}\|d\|^{2}.
\end{align*}
This is equivalent to the quadratic programming (QP):
\begin{align*}
	&\min\limits_{(t,d)\in\mathbb{R}\times\mathbb{R}^{n}} t + \frac{1}{2}\|d\|^{2},\\
	&\ \ \ \ \mathrm{ s.t.} \ \nabla F_{i}(x)^{\mathrm T}d\leqslant t,\ i=1,2,...,m.
\end{align*}
From then on, some standard mathematical programming methods are generalized to solve MOPs and vector optimization problems \cite{7,8,8a,9,9a,9b,9c,10}. Among these, Povalej \cite{8a} proposed the following direction-finding subproblem of the quasi-Newton method: 
\begin{align*}
	\min\limits_{d\in\mathbb{R}^{n}} \max\limits_{i=1,2,...,m}\ \nabla F_{i}(x)^{\mathrm T}d+\frac{1}{2}d^{\mathrm T}B_{i}(x)d,
\end{align*}
where $B_{i}(x)$ is the approximate Hessian
of objective function $F_{i}$ at $x$. Due to the coupled $\nabla F_{i}(x)$ and $B_{i}(x)$ for $i=1,2,...,m$, the subproblem can be reformed into the quadratically constrained problem (QCP):
\begin{align*}
	&\min\limits_{(t,d)\in\mathbb{R}\times\mathbb{R}^{n}} t,\\
	&\ \ \ \ \mathrm{ s.t.} \ \nabla F_{i}(x)^{\mathrm T}d+\frac{1}{2}d^{\mathrm T}B_{i}(x)d\leqslant t,\ i=1,2,...,m.
\end{align*}
On the one hand, comparing with steepest descent descent method, the rate of convergence is superlinear for quasi-Newton method. On the other hand, solving QCPs is much more time-consuming than QPs. As a result, even if the overall iteration count is reduced due to quasi-Newton directions, each iteration with solving a QCP will require more computational effort than one iteration with solving a QP. The Newton method \cite{7} for MOPs will also encounter this situation. Naturally, in view of the efficient inner iterations of the steepest descent method and rapid outer convergence of the quasi-Newton method, a question arises: can we use a common matrix as an approximation to Hessian matrices of all objective functions so that decreased direction can be obtained by solving a QP and outer iterations enjoy superlinear convergence as well?
\par Variable metric method \cite{25} is one of the main solution strategies for solving constrained optimization problems, in which the calculation of direction $d_{k}$ depends on a common positive definite matrix $B_{k}$. Particularly, $B_{k}$ is an approximation to Hessian matrices of objective function and constraints at $x_{k}$, and convergence rate of the method is superlinear in this setting. It is easy to apply this method to direction-finding subproblems of MOPs. As far as we know, some researchers have already worked on it. Moudden and Mouatasim \cite{10} used $B_{k}=\tau_{k}I$ as a diagonal approximation of Hessian matrices. On the other hand,  Ansary and Panda \cite{10a} used BFGS \cite{21,22,23,24} to update approximate Hessian matrices. However, they didn't prove the superlinear convergence rate of the proposed algorithm, and which is left as their next objective. To our best knowledge, the question is still open. In this paper, we present a variable metric method to remedy second-order methods \cite{7,8a} for unconstrained multiobjective optimization problems. We also prove the local superlinear convergence rate of the proposed method in a special setting, which is a positive answer to the open question.
\par The organization of the paper is as follows. Some preliminaries are given in Sect. 2 for our later use. In Sect. 3, we propose the generic frameworks for solving unconstrained MOPs. Sect. 4 is devoted to prove global and strong convergence of the proposed method, and show that sequence of solutions to dual problems is convergent under reasonable assumptions. Sect. 5 uses BFGS to update approximate Hessian matrices, along which, a new nonmonotone line search technique is proposed to achieve a local superlinear convergence rate. Numerical results are presented in Sect. 6. At the end of the paper, some conclusions are drawn.

\section{Preliminaries }
In this paper, the (\ref{MOP}) is considered in the $n$-dimensional real Euclidean space $\mathbb{R}^{n}$. We denote by $\mathbb{R}_{+}$ the set of nonnegative real numbers, by $\mathbb{R}_{++}$ the set of strictly positive real numbers, and by $\|\cdot\|$ the Euclidean distance in $\mathbb{R}^{n}$. Let $\textbf{B}[x,r]$ be a closed ball with radius $r\in\mathbb{R}^{+}$ and center at $x\in\mathbb{R}^{n}$. We denote by $JF(x)\in\mathbb{R}^{m\times n}$ the Jacobian matrix of $F$ at $x$, by $\nabla F_{i}(x)\in\mathbb{R}^{n}$ the gradient of $F_{i}$ at $x$ and by $\nabla^{2}F_{i}(x)\in\mathbb{R}^{n\times n}$ the Hessian matrix of $F_{i}$ at $x$. $I\in\mathbb{R}^{n\times n}$ denote unit matrix, for two matrices $A$, $B$, $B\leqslant A\ (\mathrm{resp.}\ B<A)$ means that $A-B$ is positive semidefinite (resp. definite). We denote by conv$(S)$ the convex hull of a set $S$.  Let $v$ be a vector, $c$ be a constant, $v\geqslant c$ represents every component of $v$ is greater than or equal to $c$. The Hausdorff distance between two sets is given by 
$$d_{H}(A,B):=\max\{\sup\limits_{a\in A}d(a,B),\ \sup\limits_{b\in B}d(b,A)\},$$
where $d(a,B):=\inf\limits_{b\in B}\{\|a-b\|\}$. 
To optimize $F$, we present the definition of optimal solutions in Pareto sense. We introduce partial order induced by $\mathbb{R}^{m}_{+}=\mathbb{R}_{+}\times\cdots\times\mathbb{R}_{+}$:
$$F(y)\leqslant(\mathrm{resp.}\ <)F(z)\ \Leftrightarrow\ F(z)-F(y)\in\mathbb{R}^{m}_{+}(\mathrm{resp.}\ \mathbb{R}^{m}_{++}).$$
\par Some definitions used in this paper are given below.
\begin{defi}\rm\cite{6b}
	A vector $x^{\ast}\in\mathbb{R}^{n}$ is called Pareto optimum to (\ref{MOP}), if there exists no $x\in\mathbb{R}^{n}$ such that $F(x)\leqslant F(x^{\ast})$ and $F(x)\neq F(x^{\ast})$.
\end{defi}

\begin{defi}\label{d1}\rm\cite{6}
	A vector $x^{\ast}\in\mathbb{R}^{n}$ is called  Pareto critical point to (\ref{MOP}), if
	$$\mathrm{range}(JF(x^{*}))\cap-\mathbb{R}_{++}^{m}=\emptyset,$$
	where $\mathrm{range}(JF(x^{*}))$ denotes the range of linear mapping given by the matrix $JF(x^{*})$.
\end{defi}

\begin{defi}\rm\cite{6}
	A vector $d\in\mathbb{R}^{n}$ is called descent direction for $F$ at $x$, if
	$$JF(x)d\in-\mathbb{R}_{++}^{m}.$$
\end{defi}
\section{ Generic Variable Metric Method for MOPs}
\par For $x\in\mathbb{R}^{n}$, we define $d(x)$, the descent direction obtained by variable metric method, as the optimal solution of 
\begin{align}\label{eq3.1}
	\min\limits_{d\in\mathbb{R}^{n}} \max\limits_{i=1,2,...,m}\ \nabla F_{i}(x)^{\mathrm T}d+\frac{1}{2}d^{\mathrm T}B(x)d
\end{align}
where
$B(x)$ is a positive definite matrix.
\par Since that $\nabla F_{i}(x)^{\mathrm T}d+\frac{1}{2}d^{\mathrm T}B(x)d$ is strongly convex for $i=1,2,...,m$, we conclude that (\ref{eq3.1}) has a unique minimizer. The optimal value and optimal solution of (\ref{eq3.1}) are denoted by $\theta(x)$ and $d(x)$, respectively. Hence,
\begin{equation}\label{E2.1}
	\theta(x) = \min\limits_{d\in \mathbb{R}^{n}}\max\limits_{i=1,2,...,m}\nabla F_{i}(x)^{\mathrm T}d+\frac{1}{2}d^{\mathrm T}B(x)d,
\end{equation}
and
\begin{equation}\label{E2.2}
	d(x) = \arg\min\limits_{d\in \mathbb{R}^{n}}\max\limits_{i=1,2,...,m}\nabla F_{i}(x)^{\mathrm T}d+\frac{1}{2}d^{\mathrm T}B(x)d.
\end{equation}
\par Indeed, (\ref{eq3.1}) can be rewritten equivalently as the following smooth quadratic problem:
\begin{align*}\tag{QP}\label{QP}
	&\min\limits_{(t,d)\in\mathbb{R}\times\mathbb{R}^{n}}\ t+\frac{1}{2}d^{\mathrm T}B(x)d,\\ 
	&\ \ \ \ \mathrm{ s.t.} \ \nabla F_{i}(x)^{\mathrm T}d \leqslant t,\ i=1,2,...,m.
\end{align*}
Notice that (\ref{QP}) is convex and its constraints are linear, then strong duality holds. The Lagrangian function of (\ref{QP}) is
$$L((t,d),\lambda)=t+\frac{1}{2}d^{\mathrm T}B(x)d +\sum\limits_{i=1}^{m}\lambda_{i}(\nabla F_{i}(x)^{\mathrm T}d-t).$$
By Karush-Kuhn-Tucker (KKT) conditions, we have
\begin{equation}\label{E3.3}
	\sum\limits_{i=1}^{m}\lambda_{i}=1,
\end{equation}

\begin{equation}\label{E3.4}
	B(x)d + \sum\limits_{i=1}^{m}\lambda_{i}\nabla F_{i}(x)=0,
\end{equation}

\begin{equation}\label{E3.5}
	\nabla F_{i}(x)^{\mathrm T}d \leqslant t,\ i=1,2,...,m,
\end{equation}

\begin{equation}\label{E3.6}
	\lambda_{i}\geqslant0,\ i=1,2,...,m,
\end{equation}

\begin{equation}\label{E3.7}
	\lambda_{i}(\nabla F_{i}(x)^{\mathrm T}d-t)=0,\ i=1,2,...,m.
\end{equation}
From (\ref{E3.4}) and positive definiteness of $B(x)$, we obtain
\begin{equation}\label{E3.8}
	d(x) = -B(x)^{-1}(\sum\limits_{i=1}^{m}\lambda_{i}(x)\nabla F_{i}(x)),
\end{equation}
where $\lambda(x)=(\lambda_{1}(x),\lambda_{2}(x),...,\lambda_{m}(x))$ is the solution of dual problem:
\begin{align*}\tag{DP}\label{DP}
	 -&\min\limits_{\lambda}\frac{1}{2} (\sum\limits_{i=1}^{m}\lambda_{i}\nabla F_{i}(x))^{\mathrm T}B(x)^{-1}(\sum\limits_{i=1}^{m}\lambda_{i}\nabla F_{i}(x))\\
	&\mathrm{ s.t.} \ \sum\limits_{i=1}^{m}\lambda_{i}=1,\\
	&\ \ \ \ \ \ 	\lambda_{i}\geqslant0,\ i=1,2,...,m.	
\end{align*}
For the sake of simplicity, we denote $\Delta_{m}=\{\lambda:\sum\limits_{i=1}^{m}\lambda_{i}=1,\lambda_{i}\geqslant0,\ i=1,2,...,m\}$.
Recall that strong duality holds, we obtain
\begin{equation}\label{E3.9}
	\theta(x)=-\frac{1}{2} (\sum\limits_{i=1}^{m}\lambda_{i}(x)\nabla F_{i}(x))^{\mathrm T}B(x)^{-1}(\sum\limits_{i=1}^{m}\lambda_{i}(x)\nabla F_{i}(x))=-\frac{1}{2}d(x)^{\mathrm T}B(x)d(x).
\end{equation}
From (\ref{E3.5}), we have
\begin{equation}\label{E3.10}
	\nabla F_{i}(x)^{\mathrm T}d(x) \leqslant t(x)=\theta(x)-\frac{1}{2}d(x)^{\mathrm T}B(x)d(x)=-d(x)^{\mathrm T}B(x)d(x),\ i=1,2,...,m.
\end{equation}
If $\lambda_{i}(x)\neq0$, then (\ref{E3.7}) leads to
\begin{equation}\label{e3.10}
	\nabla F_{i}(x)^{\mathrm T}d(x) = t(x)=-d(x)^{\mathrm T}B(x)d(x).
\end{equation}
It follows by (\ref{E3.3}) and (\ref{E3.7}) that
\begin{equation}\label{E3.11}
	\sum\limits_{i=1}^{m}\lambda_{i}(x)\nabla F_{i}(x)^{\mathrm{T}}d(x)= t(x)=-d(x)^{\mathrm T}B(x)d(x).
\end{equation}
\par Next, we will present some properties of $\theta(x)$ and $d(x)$.
\begin{lem}\rm\label{l3.1}
	Suppose that $B(x)$ is positive definite for all $x\in\mathbb{R}^{n}$, we have
	\begin{itemize}
		\item[$\mathrm{(a)}$] for all $x\in\mathbb{R}^{n}$, $\theta(x)\leqslant0$.
		\item[$\mathrm{(b)}$] the following conditions are equivalent:
		\subitem$\mathrm{(i)}$ The point $x$ is non-critical;
		\subitem$\mathrm{(ii)}$ $\theta(x)<0$;
		\subitem$\mathrm{(iii)}$ $d(x)\neq0.$
		\item[$\mathrm{(c)}$] if $B:$ $\mathbb{R}^{n}\rightarrow\mathbb{R}^{n\times n}$ is continuous, then so are $d$ and $\theta$.
	\end{itemize}
\end{lem}
\begin{proof}
	Since matrix $B(x)$ is positive definite, then assertions (a) and (b) can be obtained by using the same arguments as in the proof of \cite[Lemma 3.2]{8a}. Next, we prove the assertion (c). It is sufficient to prove that $d$ and $\theta$ are continuous in a fixed but arbitrary compact set $W\subset\mathbb{R}^{n}$. Notice that $B$ and $\nabla F_{i}$ are continuous, and the fact that $B(x)$ is positive definite, then $\|\nabla F_{i}\|$ is bounded and eigenvalues of $B$ are uniformly bounded away from $0$ on $W$, i.e., there exist $L$ and $a,b>0$ such that
	\begin{equation}\label{E3.12}
		L=\max\limits_{x\in W,\ i=1,2,...,m}\|\nabla F_{i}(x)\|
	\end{equation}
	and 
	\begin{equation}\label{E3.13}
		aI\leqslant B(x)\leqslant bI, \ \forall x\in W.
	\end{equation}
	In view of (\ref{E3.8}), we have
	\begin{equation}\label{e3.14}
		d(x)=-B(x)^{-\frac{1}{2}}(B(x)^{-\frac{1}{2}}(\sum\limits_{i=1}^{m}\lambda_{i}(x)\nabla F_{i}(x))),
	\end{equation}
	where $\lambda(x)$ is a solution of (\ref{DP}). Let us first define $C(x)=\mathrm{conv}(\{B(x)^{-\frac{1}{2}}\nabla F_{i}(x):\ i=1,2,...,m\})$ and $\zeta(x)=B(x)^{-\frac{1}{2}}(\sum\limits_{i=1}^{m}\lambda_{i}(x)\nabla F_{i}(x))$, then $\zeta(x)$ is the minimal norm element of $C(x)$ due to $\lambda(x)$ is a solution of (\ref{DP}). Together with (\ref{E3.13}) and (\ref{e3.14}), to prove $d$ is continuous, it is sufficient to show that $\zeta$ is continuous. Next, we will prove that, for a fixed point $y\in W$ and any $\epsilon>0$, there exists $\delta>0$ such that
	\begin{equation}\label{E3.14}
		d_{H}(C(y),C(z))<\epsilon,\ \ \forall z\in \textbf{B}[y,\delta].
	\end{equation}
	Since $B(x)^{-\frac{1}{2}}\nabla F_{i}(x)$ is continuous for $i=1,2,...,m$, without loss of generality, for any $\epsilon>0$, there exists $\delta>0$ such that
	$$\|B(y)^{-\frac{1}{2}}\nabla F_{i}(y)-B(z)^{-\frac{1}{2}}\nabla F_{i}(z)\|<\epsilon,\ \forall z\in \textbf{B}[y,\delta],\ i=1,2,...,m.$$ 
	For any $c\in C(y)$, there exists $\lambda^{c}\in\Delta_{m}$ such that $c =\sum\limits_{i=1}^{m}\lambda_{i}^{c}B(y)^{-\frac{1}{2}}\nabla F_{i}(y)$. Hence, 
	\begin{align*}
		d(c,C(z))\leqslant&\ \|\sum\limits_{i=1}^{m}\lambda_{i}^{c}B(y)^{-\frac{1}{2}}\nabla F_{i}(y)-\sum\limits_{i=1}^{m}\lambda_{i}^{c}B(z)^{-\frac{1}{2}}\nabla F_{i}(z)\|\\
		\leqslant&\ \sum\limits_{i=1}^{m}\lambda_{i}^{c}\|B(y)^{-\frac{1}{2}}\nabla F_{i}(y)-B(z)^{-\frac{1}{2}}\nabla F_{i}(z)\|\\
		<&\ \epsilon
	\end{align*}
	holds for all $z\in\textbf{B}[y,\delta]$. Since $c\in C(y)$ is arbitrary, we have
	$$\sup\limits_{c\in C(y)}d(c,C(z))<\epsilon, \ \ \forall z\in \textbf{B}[y,\delta].$$
	Similarly, we obtain that
	$$\sup\limits_{c\in C(z)}d(c,C(y))<\epsilon, \ \ \forall z\in \textbf{B}[y,\delta].$$
	Therefore, inequality (\ref{E3.14}) holds.
	According to the definition of $\zeta$, it follows that
	$$c_{y}^{\mathrm T}\zeta(y)\geqslant\|\zeta(y)\|^{2},\ \ \forall c_{y}\in C(y),$$
	and
	$$c_{z}^{\mathrm T}\zeta(z)\geqslant\|\zeta(z)\|^{2},\ \ \forall c_{z}\in C(z).$$
	This together with (\ref{E3.14}) implies that there exist $w_{y},w_{z}\in \textbf{B}[0,\epsilon]$ such that
	$$(\zeta(z) + w_{z})^{\mathrm T}\zeta(y)\geqslant\|\zeta(y)\|^{2},$$
	and 
	$$(\zeta(y) + w_{y})^{\mathrm T}\zeta(z)\geqslant\|\zeta(z)\|^{2}.$$
	Consequently, we infer that
	\begin{align*}
		\|\zeta(y)-\zeta(z)\|^{2}&=\|\zeta(y)\|^{2}+\|\zeta(z)\|^{2}-2\zeta(y)^{\mathrm T}\zeta(z)\\
		&\leqslant w_{z}^{\mathrm T}\zeta(y)+w_{y}^{\mathrm T}\zeta(z)\\
		&\leqslant\epsilon(\|\zeta(y)\|+\|\zeta(z)\|)\\
		&\leqslant\frac{2L}{\sqrt{a}}\epsilon,
	\end{align*}
	where the last inequality is given by (\ref{E3.12}), (\ref{E3.13}) and the definition of $C$. Thus, we have proved that $\zeta$ is continuous in $W$, i.e., $d$ is continuous. Finally, the continuity of $\theta$ is the consequence of (\ref{E3.9}) and the fact that $d$ and $B$ are continuous. 
\end{proof}
\par For every iteration $k$, after obtaining the unique descent direction $d_{k}\neq0$, the classical Armijo technique is applied for line search. 
\begin{algorithm}[H]  
	\caption{\ttfamily Armijo\_line\_search} 
	\KwData{$x_{k}\in\mathbb{R}^{n},d_{k}\in\mathbb{R}^{n},JF(x_{k})\in\mathbb{R}^{m\times n},\sigma,\gamma\in(0,1),\alpha=1$}
	\While{$F(x_{k}+\alpha d_{k})- F(x_{k}) \nleq \sigma\alpha JF(x_{k})d_{k}$}{~\\$\alpha\leftarrow \gamma\alpha$ }\Return{$\alpha_{k}\leftarrow \alpha$}  
\end{algorithm}
\par The following result shows that the Armijo technique will accept a step size along with $d(x)$.
\begin{lem}\rm\label{l3.2}\cite[Lemma 4]{6}
	Assume that $x\in\mathbb{R}^{n}$ is a noncritical point of $F$. Then, for any $\sigma\in(0,1)$ there exists $\alpha_{0}\in (0,1]$ such that
	$$F_{i}(x+\alpha d(x)) - F_{i}(x)\leqslant \sigma \alpha JF(x)d(x)$$
	holds for all $\alpha\in [0,\alpha_{0}]$ and $i\in\{1,2,...,m\}$.
\end{lem}
The generic variable metric method (GVMM) for MOPs is described as follows.
\begin{algorithm}[H]  
	\caption{\ttfamily \ttfamily generic\_variable\_metric\_method\_for\_MOPs}
	\KwData{$x_{0}\in\mathbb{R}^{n}$, $\varepsilon\geqslant0$, positive definite martix $B_{0}\in\mathbb{R}^{n\times n}$}
	\For{$k=0,1,...$}{  $\lambda^{k}\leftarrow \arg\min\limits_{\lambda\in\Delta_{m}}\frac{1}{2} (\sum\limits_{i=1}^{m}\lambda_{i}\nabla F_{i}(x_{k}))^{\mathrm T}B_{k}^{-1}(\sum\limits_{i=1}^{m}\lambda_{i}\nabla F_{i}(x_{k}))$\\  
		$d_{k}\leftarrow-B_{k}^{-1}(\sum\limits_{i=1}^{m}\lambda_{i}^{k}\nabla F_{i}(x_{k}))$\\
		$\theta_{k}\leftarrow-\frac{1}{2}d_{k}^{\mathrm T}(\sum\limits_{i=1}^{m}\lambda^{k}_{i}\nabla F_{i}(x_{k}))$\\
		\eIf{$|\theta_{k}|\leqslant\varepsilon$}{~\\ {\bf{return}} $x_{k}$ }{$\alpha_{k}\leftarrow$ {\ttfamily Armijo\_line\_search}$(x_{k},d_{k},JF(x_{k}))$\\
			$x_{k+1}\leftarrow x_{k}+\alpha_{k}d_{k}$\\
			update $B_{k+1}$  }}  
\end{algorithm}
\begin{rem}
	In Algorithm 2, we obtain the descent direction by solving (\ref{DP}) rather than solving (\ref{QP}). Indeed, as the dimension of decision variables grows, the effort to solve (\ref{QP}) increases. However, the dimension of dual variables, which is equal to the number of objectives, is invariable as the dimension of decision variables grows.
\end{rem}
\section{Convergence Analysis for Generic Variable Metric Method}
In Algorithm 2, we can see that GVMM terminates with an $\varepsilon$-approximate Pareto critical point in a finite number of iterations or generates an infinite sequence. In the sequel, we will suppose that GVMM produces an infinite sequence of noncritical points. The main goal of this section is to show the convergence property of GVMM.
\subsection{Global Convergence}
In this subsection, we will discuss the global convergence of GVMM. First of all, some mild assumptions are presented as follows.
\begin{assumption}\label{a1}
	The sequence $\{B_{k}\}$ of matrices is uniformly positive definite, i.e., there exist two positive constants, $a$ and $b$, such that
	$$aI\leqslant B_{k}\leqslant bI,\ \forall k.$$
\end{assumption}
\begin{assumption}\label{a2}
	For any $x_{0}\in\mathbb{R}^{n}$, the level set $L=\{x:\ F(x)\leqslant F(x_{0})\}$ is compact.
\end{assumption}
\begin{theo}\rm\label{t4.1}
	Suppose that Assumptions \ref{a1} and \ref{a2} hold, sequence $\{x_{k}\}$ is produced by GVMM, then we have
	\begin{itemize}
		\item[$\mathrm{(i)}$] $\lim\limits_{k\rightarrow\infty}x_{k+1} - x_{k}=0$.
		\item[$\mathrm{(ii)}$] $\{x_{k}\}$ has at least one accumulation point, and every accumulation point is a Pareto critical point.
		\item[$\mathrm{(iii)}$] $\lim\limits_{k\rightarrow\infty}d_{k}=0.$
	\end{itemize}
\end{theo}
\begin{proof}
	Since $L=\{x\in\mathbb{R}^{n}|F(x)\leqslant F(x_{0})\}$ is compact and $\{F(x_{k})\}$ is a decreasing sequence, it follows that sequence $\{x_{k}\}$ belongs to $L$, and $\{x_{k}\}$ has at least one accumulation point. From $L$ is compact and $F$ is continuous, we conclude that $\{F(x_{k})\}$ is bounded, which together with the monotonicity $\{F(x_{k})\}$ yields $\{F(x_{k})\}$ is a cauthy sequence. Therefore,
	\begin{equation}\label{e1}
		\lim\limits_{k\rightarrow\infty}(F(x_{k+1})-F(x_{k}))=0.
	\end{equation}
	Recall that $\alpha_{k}$ is chosen as follows:
	\begin{equation}\label{e2}
		F(x_{k}+\alpha_{k}d_{k})\leqslant F(x_{k}) + \sigma \alpha_{k}JF(x_{k})d_{k},
	\end{equation}
	we use Assumption \ref{a1} and (\ref{E3.10}) to obtain
	\begin{equation}\label{e3}
		F(x_{k}+\alpha_{k}d_{k})\leqslant F(x_{k})-a\sigma\alpha_{k}\| d_{k}\|^{2}.
	\end{equation}
	The latter combined with (\ref{e1}) implies
	\begin{equation}\label{e4}
		\lim\limits_{k\rightarrow\infty}(\alpha_{k}\| d_{k}\|^{2})=0.
	\end{equation}
	Therefore, we have
	$$\lim\limits_{k\rightarrow\infty}x_{k+1} - x_{k}=\lim\limits_{k\rightarrow\infty}\alpha_{k}d_{k}=0,$$
	which completes the proof of assertion (i). Next, we will prove the assertion (ii). Since $\{x_{k}\}$ is bounded and $\{B_{k}\}$ is uniformly positive definite, without loss of generality, there exist an infinite index set $K$, a point $x^{*}$ and a positive definite matrix $B^{*}$ such that
	$$x_{k}\rightarrow x^{*},\ B_{k}\rightarrow B^{*},\ k\in K.$$ 
	In view of (\ref{e4}), we distinguish two cases:
	\begin{equation}\label{Eq4.1}
		\limsup\limits_{k\in K}\alpha_{k}>0
	\end{equation}
	and 
	\begin{equation}\label{Eq4.2}
		\limsup\limits_{k\in K}\alpha_{k}=0.
	\end{equation}
	If (\ref{Eq4.1}) holds, then there exists an infinite index set $K_{1}\subset K$ such that
	$$\lim\limits_{k_{1}\in K_{1}}\alpha_{k_{1}}=\bar{\alpha}>0.$$
	This together with (\ref{e4}) implies
	$$\lim\limits_{k_{1}\in K_{1}}d_{k_{1}}=0.$$
	Recall that $B_{k_{1}}\rightarrow B^{*}$ and $x_{k_1}\rightarrow x^{*}$, it follows from Lemma \ref{l3.1} (c) that $d_{k_{1}}\rightarrow d^{*}$, i.e., $d^{*}=0$. Hence, we conclude that $x^{*}$ is a Pareto critical point.
	\par When (\ref{Eq4.2}) occurs, it follows by $\alpha_{k}\geqslant0$ that $\lim\limits_{k\in K}\alpha_{k}=0$. From the Armijo line search, without loss of generality, we have  
	$$F(x_{k}+\frac{\alpha_{k}}{\gamma} d_{k})- F(x_{k}) \nleqslant \sigma\frac{\alpha_{k}}{\gamma} JF(x_{k})d_{k},\ \forall k\in K.$$
	This implies that there exists at least one $j_{k}\in [m]$ such that
	$$F_{j_{k}}(x_{k}+\frac{\alpha_{k}}{\gamma} d_{k})- F_{j_{k}}(x_{k})>\sigma\frac{\alpha_{k}}{\gamma} \nabla F_{j_{k}}(x_{k})^{\mathrm{T}}d_{k},\ \forall k\in K.$$
	Dividing by $\alpha_{k}$ on both sides of the above inequality and let $k\in K$ tends to $\infty$, we obtain
	$$\nabla F_{i}(x^{*})^{\mathrm{T}}d^{*}\geqslant0,$$
	for some $i\in[m]$. It follows by (\ref{E3.10}) that $-d^{*\mathrm{T}}B^{*}d^{*}\geqslant0$, i.e., $d^{*}=0$. Hence, $x^{*}$ is a Pareto critical point.
	\par Finally, we prove the assertion (iii). Suppose by contradiction that there exists an infinite index set $\bar{K}$ and a constant $c>0$ such that $\|d_{k}\|\geqslant c$ holds for all $k\in\bar{K}$. Since $\{x_{k}\}$ is bounded and $\{B_{k}\}$ is uniformly positive definite, without loss of generality, there exist an infinite index set $\bar{K}^{\prime}\subset\bar{K}$, a point $\bar{x}$ and a positive definite matrix $\bar{B}$ such that
	$$x_{k^{\prime}}\rightarrow \bar{x},\ B_{k^{\prime}}\rightarrow\bar{B},\ k^{\prime}\in\bar{K}^{\prime}.$$  
	The latter combined with assertion (c) of Lemma \ref{l3.1} yields $\lim\limits_{k^{\prime}\in\bar{K}^{\prime}}\|d_{k^{\prime}}\|=\|\bar{d}\|\geqslant c>0$. Therefore, by Lemma \ref{l3.1} (b), we conclude that $\bar{x}$ is not a Pareto critical point. Recall that $\bar{x}$ is an accumulation point of $\{x_{k}\}$, this contradicts assertion (ii). Thus, $\lim\limits_{k\rightarrow\infty}d_{k}=0.$ The proof is completed. 
\end{proof} 
\subsection{Strong Convergence}
This subsection is devoted to strong convergence of the proposed method. On the one hand, it should be noted that although the strong convergence property of an algorithm has been extensively studied in SOPs, the researches on the counterpart for MOPs are relatively rare. On the other hand, a question arises: will the gradient-based method converge to a Pareto critical point without convexity assumption? In order to analyze strong convergence property of GVMM, a second-order assumption is presented as follows.
\begin{assumption}\label{a3}
	The sequence $\{x_{k}\}$ produced by GVMM possesses an accumulation point $x^{*}$ with the corresponding multiplier $\lambda^{*}$ such that
	$$d^{\mathrm T}(\sum\limits_{i=1}^{m}\lambda^{*}_{i}\nabla^{2}F_{i}(x^{*}))d>0, \ \ \forall d\in\Omega,$$ where $\Omega:=\{d\neq0: \nabla F_{i}(x^{*})^{\mathrm T}d=0,\ i=1,2,...,m\}$.
\end{assumption}

\begin{rem}
	It is easy to see that convex multiobjective function $F$ with at least one strictly convex component $F_{i}$, and $\lambda^{*}_{i}\neq0$, is sufficient for the above assumption. Moreover, the assumption holds for some nonconvex functions when $\sum\limits_{i=1}^{m}\lambda^{*}_{i}\nabla^{2}F_{i}(x^{*})$ is positive definite. We will give an example to verify it.
\end{rem}

\begin{exam}
	Consider the following multiobjective optimization:
	$$\min\limits_{x\in\mathbb{R}^{2}}(F_{1}(x),F_{2}(x)),$$
	where $F_{1}(x_{1},x_{2})=2x^{2}_{1}-x^{2}_{2}$, $F_{2}(x_{1},x_{2})=-(x_{1}-1)^{2}+2(x_{2}-1)^{2}$. By directly calculating, we have
	$$\nabla F_{1}(x_{1},x_{2})=(4x_{1},-2x_2)^{\mathrm{T}},$$
	$$\nabla F_{2}(x_{1},x_{2})=(-2(x_{1}-1),4(x_2-1))^{\mathrm{T}},$$
	and $$\nabla^{2} F_{1}(x_{1},x_{2})=\begin{pmatrix} 4\ &\ 0\\0\ &\ -2 \end{pmatrix},$$
	$$\nabla^{2} F_{2}(x_{1},x_{2})=\begin{pmatrix} -2\ &\ 0\\0\ &\ 4 \end{pmatrix}.$$
	Substituting $x=(-1,2)^{\mathrm{T}}$ into gradient functions, we have
	$\nabla F_{1}(x)=(-4,-4)^{\mathrm{T}}$, and $\nabla F_{2}(x)=(4,4)^{\mathrm{T}}$. It follows that $\frac{1}{2}\nabla F_{1}(x)+\frac{1}{2}\nabla F_{2}(x)=0$. 
	Note that $F_{1}$ and $F_{2}$ are nonconvex, but $\frac{1}{2}\nabla^{2}F_{1}(x)+\frac{1}{2}\nabla^{2}F_{2}(x)=\begin{pmatrix} 1\ &\ 0\\0\ &\ 1 \end{pmatrix}$ is positive definite. Then Assumption \ref{a3} holds for the nonconvex function $F$.
\end{exam}

\begin{theo}\rm\label{t4.2}
	Suppose that Assumptions \ref{a1}, \ref{a2} and \ref{a3} hold, then $\{x_{k}\}$ produced by GVMM is strongly convergent, i.e., $\lim\limits_{k\rightarrow\infty}x_{k}=x^{*}$.  
\end{theo}
\begin{proof} By Theorem \ref{t4.1} (ii), $\{x_{k}\}$ has at least one  accumulation point, and the accumulation point $x^{*}$ is a Pareto critical point. Next, we prove that $x^{*}$ is an isolated accumulation point of $\{x_{k}\}$. Assume in contradiction that $x^{*}$ is not an isolated accumulation point of $\{x_{k}\}$. Then there exists a sequence $\{x_{j}^{*}\}$ satisfies $x_{j}^{*}\neq x^{*}$ such that $\lim\limits_{j\rightarrow\infty}x_{j}^{*}=x^{*}$ and every $x_{j}^{*}$ is an accumulation point of the sequence $\{x_{k}\}$. Since $x^{*}$ is a Pareto critical point, then there exists a vector $\lambda^{*}\in\Delta_{m}$ such that
	\begin{equation}\label{E4.3}
		\sum\limits_{i=1}^{m}\lambda^{*}_{i}\nabla F_{i}(x^{*})=0.
	\end{equation}
	Without loss of generality, we assume that
	\begin{equation}\label{E4.1}
		\lim\limits_{j\rightarrow\infty}\frac{x_{j}^{*}-x^{*}}{\|x_{j}^{*}-x^{*}\|}=d.
	\end{equation}
	Since $\{F(x_{k})\}$ converges to $F^{*}$, then every accumulation point of the sequence $\{x_{k}\}$ has the same value for $F$, that is, 
	\begin{equation}\label{E4.2}
		F(x_{j}^{*})=F(x^{*}),\ \ \forall j. 
	\end{equation}
	Hence, 
	$$0=F_{i}(x_{j}^{*})-F_{i}(x^{*})=\nabla F_{i}(x^{*})^{\mathrm T}(x_{j}^{*}-x^{*})+o(\|x_{j}^{*}-x^{*}\|),\ \ \forall i =1,2,...,m,\ j\rightarrow\infty.$$
	Divide by $\|x_{j}^{*}-x^{*}\|$ on both sides of the above equalities, as $j$ tends to $\infty$, we use (\ref{E4.1}) and (\ref{E4.2}) to get
	$$\nabla F_{i}(x^{*})^{\mathrm T}d=0,\ \ \forall i.$$
	On the other hand, by virtue of (\ref{E4.3}) and (\ref{E4.2}), we have
	\begin{align*}
		0=&\sum\limits_{i=1}^{m}\lambda^{*}_{i}F_{i}(x_{j}^{*}) - \sum\limits_{i=1}^{m}\lambda_{i}^{*}F_{i}(x^{*})\\
		=&(\sum\limits_{i=1}^{m}\lambda_{i}^{*}\nabla F_{i}(x^{*}))^{\mathrm T}(x_{j}^{*}-x^{*}) +\frac{1}{2}(x_{j}^{*}-x^{*})^{\mathrm T}(\sum\limits_{i=1}^{m}\lambda_{i}^{*}\nabla^{2}F_{i}(x^{*}))(x_{j}^{*}-x^{*})+o(\|x_{j}^{*}-x^{*}\|^{2})\\
		=&\frac{1}{2}(x_{j}^{*}-x^{*})^{\mathrm T}(\sum\limits_{i=1}^{m}\lambda^{*}_{i}\nabla^{2}F_{i}(x^{*}))(x_{j}^{*}-x^{*})+o(\|x_{j}^{*}-x^{*}\|^{2}),\ j\rightarrow\infty.
	\end{align*}
	Hence, one has
	$$d^{\mathrm T}(\sum\limits_{i=1}^{m}\lambda^{*}_{i}\nabla^{2}F_{i}(x^{*}))d=0,$$
	which is a contradiction to Assumption \ref{a3}. Therefore, $x^{*}$ is an isolated accumulation point of $\{x_{k}\}$, which together with Theorem \ref{t4.1} (i) yields $\lim\limits_{k\rightarrow\infty}x_{k}=x^{*}$ (see \cite[Theorem 1.1.5]{11}).
\end{proof}
\begin{rem}
	Particularly, if we set $B_{k}=I$, then Theorem \ref{t4.2} is valid in steepest descent method \rm{\cite{6}}. 
\end{rem}

In Lemma \ref{l3.1}, we can not obtain the continuity of $\lambda$, the solution to (\ref{DP}). The following result presents the convergence property of $\{\lambda_{k}\}$ under a reasonable assumption.
\begin{assumption}\label{a4}
	$\{\nabla F_{i}(x^{*}):i\in \mathcal{A}(x^{*})\}$ are affinely independent, where $\mathcal{A}(x^{*}):=\{i:\nabla F_{i}(x^{*})^{\mathrm T}d(x^{*})=t(x^{*})\}$ is the set of active constraints at $x^{*}$.
\end{assumption}
\begin{rem}
	Assumption \ref{a4} is equivalent to the condition that $\{\nabla F_{i}(x^{*})-\nabla F_{i_{0}}(x^{*}):\ i\in \mathcal{A}(x^{*})\setminus i_{0}\}$ are linearly independent for any fixed $i_{0}\in \mathcal{A}(x^{*})$. Thus, Assumption \ref{a4} holds automatically for bi-objective optimization problems.
\end{rem}
\begin{theo}\rm\label{t4.3}
	Suppose that Assumptions \ref{a1}, \ref{a2}, \ref{a3} and \ref{a4} hold, then $\lim\limits_{k\rightarrow\infty}\lambda_{k}=\lambda^{*}$.  
\end{theo}
\begin{proof} From Theorem \ref{t4.2}, we have
	$$\lim\limits_{k\rightarrow\infty}x_{k}=x^{*},$$
	and $x^{*}$ is a Pareto critical point. Then there exists a vector $\lambda^{*}\in\Delta_{m}$ such that
	\begin{equation}\label{E4.5}
		\sum\limits_{i=1}^{m}\lambda^{*}_{i}\nabla F_{i}(x^{*})=0.
	\end{equation}
	Combining Theorem \ref{t4.1} (iii) with (\ref{E3.8}), it follows that
	$$\lim\limits_{k\rightarrow\infty}d_{k}=\lim\limits_{k\rightarrow\infty}-B_{k}^{-1}(\sum\limits_{i=1}^{m}\lambda^{k}_{i}\nabla F_{i}(x_{k}))=0\ \mathrm{and}\ \lambda^{k}\in\Delta_{m}\ \mathrm{for\ all}\ k.$$
	From the fact that $B_{k}$ is uniformly positive definite, we have
	\begin{equation}\label{E4.6}
		\lim\limits_{k\rightarrow\infty}\sum\limits_{i=1}^{m}\lambda^{k}_{i}\nabla F_{i}(x_{k})=0,\ \mathrm{and}\ \lambda^{k}\in\Delta_{m}\ \mathrm{for\ all}\ k.
	\end{equation}
	Recall that $\lim\limits_{k\rightarrow\infty}\nabla F_{i}(x_{k})=\nabla F_{i}(x^{*})$ for all $i=1,2,...,m$, and $\{\nabla F_{i}(x^{*}):\ i\in \mathcal{A}(x^{*})\}$ are affinely independent, we apply (\ref{E4.5}) and (\ref{E4.6}) to imply $$\lim\limits_{k\rightarrow\infty}\lambda^{k}=\lambda^{*}.$$ 
\end{proof}
\section{Local Superlinear Convergence}
We are now in a position to consider local superlinear convergence of the sequence $\{x_{k}\}$, which depends on the particular updating of  $B_{k}$. To approximate Hessian matrices, $B_{k}$ is updated by BFGS formulation, namely,
\begin{equation}\label{E4.12}
	B_{k+1}=\left\{
	\begin{aligned}
		\begin{split}
			&B_{k}-\frac{B_{k}s_{k}s_{k}^{\mathrm T}B_{k}}{s_{k}^{\mathrm T}B_{k}s_{k}}+\frac{y_{k}y_{k}^{\mathrm T}}{s_{k}^{\mathrm T}y_{k}},\ s_{k}^{\mathrm T}y_{k}>0, \\
			&B_{k}, \ \ \ \ \ \ \ \ \ \ \ \ \ \ \ \ \ \ \ \ \ \ \ \ \ \ \ \ \text{otherwise},\\
		\end{split}
	\end{aligned}
	\right.
\end{equation}
where $s_{k}=x_{k+1}-x_{k}$, $y_{k}=\sum\limits_{i=1}^{m}\lambda^{k}_{i}(\nabla F_{i}(x_{k+1})-\nabla F_{i}(x_{k}))$. In order to slove (\ref{DP}), we require the corresponding updating of $B_{k+1}^{-1}$, i.e.,
\begin{equation}\label{E4.13}
	B_{k+1}^{-1}=\left\{
	\begin{aligned}
		&(I-\frac{s_{k}y_{k}^{\mathrm T}}{s_{k}^{\mathrm T}y_{k}})B_{k}^{-1}(I-\frac{y_{k}s_{k}^{\mathrm T}}{s_{k}^{\mathrm T}y_{k}})+\frac{s_{k}s_{k}^{\mathrm T}}{s_{k}^{\mathrm T}y_{k}},\ s_{k}^{\mathrm T}y_{k}>0, \\
		&B_{k}^{-1}, \ \ \ \ \ \ \ \ \ \ \ \ \ \ \ \ \ \ \ \ \ \ \ \ \ \ \ \ \ \ \ \ \ \ \ \ \ \ \ \text{otherwise}.\\
	\end{aligned}
	\right.
\end{equation}
Usually, $s_{k}^{\mathrm T}y_{k}>0$ is called the curvature condition, which guarantees positive definiteness of updated $B_{k+1}$.
\par We require some assumption on the resemblance between the approximate
Hessians $B_{k}$ and the true Hessians $\sum\limits_{i=1}^{m}\lambda^{k}_{i}\nabla^{2}F_{i}(x_{k})$ as the iterates converge to a local Pareto point, which is presented as follows:
\begin{equation}\label{E5}
	\lim\limits_{k\rightarrow\infty}\frac{\|(B_{k}-\sum\limits_{i=1}^{m}\lambda^{k}_{i}\nabla^{2}F_{i}(x_{k}))d_{k}\|}{\|d_{k}\|}=0.
\end{equation}

\begin{rem}
	For the BFGS in numerical optimization, it has been shown in \cite{qn} that 
	$$\lim\limits_{k\rightarrow\infty}\frac{\|(B_{k}-\nabla^{2}f(x^{*}))s_{k}\|}{\|s_{k}\|}=0.$$
	Note that $y_{k}=\sum\limits_{i=1}^{m}\lambda^{k}_{i}(\nabla F_{i}(x_{k+1})-\nabla F_{i}(x_{k}))$ in (\ref{E4.12}) and (\ref{E4.13}). Then equality (\ref{E5}) can be viewed as an aggregated scheme of the above equality. It is worth noting that equality (\ref{E5}) is a general assumption in proving superlinear convergence.
\end{rem}

\par To prove the local superlinear convergence property of the variable metric method with BFGS updates, it is sufficient to show that $d_{k}$ satisfies
$$\lim\limits_{k\rightarrow\infty}\frac{\|x_{k}+d_{k}-x^{*}\|}{\|x_{k}-x^{*}\|}=0$$
and line search condition will allow $\alpha_{k}=1$ when $k$ is sufficiently large. Unfortunately, even assumption (\ref{E5}) holds and Hessian matrices are positive definite for all objective functions, the line search condition still can not guarantee $\alpha_{k}=1$. In what follows, we give an example to claim it.
\begin{exam}\label{ex1}
	Consider the following multiobjective optimization:
	$$\min\limits_{x\in\mathbb{R}^{2}}\ (f_1(x),f_2(x))$$
	where $f_{1}(x_{1},x_{2})=\frac{1}{100}(x_{1}^{2}+x_{2}^{2}$), $f_2(x_{1},x_{2})=(x_{1}-2)^{2}+(x_{2}-2)^{2}$. Then, we have $$\nabla f_{1}(x_{1},x_{2})=\frac{1}{100}(2x_{1},2x_{2})^{\mathrm T},$$ $$\nabla f_{2}(x_{1},x_{2})=(2x_{1}-4,2x_{2}-4)^{\mathrm T},$$
	and
	$$\nabla^{2} f_{1}(x_{1},x_{2})=\begin{pmatrix} \frac{2}{100}\ &\ 0\\0\ &\ \frac{2}{100} \end{pmatrix},$$
	$$\nabla^{2}f_{2}(x_{1},x_{2})=\begin{pmatrix} 2\ &\ 0\\0\ &\ 2 \end{pmatrix}.$$
	When $x^{*}=(1,1)^{\mathrm T}$, we have
	$$\frac{100}{101}\nabla f_{1}(1,1)+\frac{1}{101}\nabla f_{2}(1,1)=(0,0)^{\mathrm T},$$
	so that $x^{*}$ is a Pareto critical point, and the corresponding multiplier $\lambda^{*}=(\frac{100}{101},\frac{1}{101})$. Hence,
	$$\sum\limits_{i=1}^{2}\lambda^{*}_{i}\nabla^{2} f_{i}(x^{*})=\begin{pmatrix}\frac{4}{101}\ &\ 0\\0\ &\ \frac{4}{101} \end{pmatrix}.$$
	Assume that there exists a sequence $\{x_{k}\}$ produced by GVMM such that $\lim\limits_{k\rightarrow\infty}x_{k}=x^{*}$ and $\lim\limits_{k\rightarrow\infty}d_{k}=0$, which combined with Theorem \ref{t4.3} implies 
	$\lim\limits_{k\rightarrow\infty}\lambda_{k}=\lambda^{*}$. It follows by (\ref{E5}) that 
	\begin{equation}\label{E5.4}
		\lim\limits_{k\rightarrow\infty}\frac{|d_{k}^{\mathrm T}(B_{k}-\sum\limits_{i=1}^{m}\lambda^{*}_{i}\nabla^{2}f_{i}(x^{*}))d_{k}|}{\|d_{k}\|^{2}}=0.
	\end{equation}
	By Taylor's expansion, we have   
	\begin{align*}
		f_{2}(x_{k}+d_{k}) - f_{2}(x_{k})&\geqslant \nabla f_{2}(x_{k})^{\mathrm T}d_{k} + \frac{1}{2}d_{k}^{\mathrm T}\nabla^{2} f_{2}(x_{k})d_{k} - \frac{\epsilon_{k}}{2}\|d_{k}\|^{2},\ \epsilon_{k}\rightarrow0.
	\end{align*} 
	If $\alpha_{k}=1$ holds in line search, then
	$$\nabla f_{2}(x_{k})^{\mathrm T}d_{k} + \frac{1}{2}d_{k}^{\mathrm T}\nabla^{2} f_{2}(x_{k})d_{k} - \frac{\epsilon_{k}}{2}\|d_{k}\|^{2}\leqslant\sigma\theta_{k}.$$
	This, combined with (\ref{E3.9}) and (\ref{e3.10}), implies
	\begin{align*}
		d_{k}^{\mathrm T}\nabla^{2} f_{2}(x_{k})d_{k}&\leqslant(2-\sigma)d_{k}^{\mathrm T}B_{k}d_{k}+\epsilon_{k}\|d_{k}\|^{2}\\
		&\leqslant (2-\sigma)(d_{k}^{\mathrm T}(\sum\limits_{i=1}^{m}\lambda^{*}_{i}\nabla^{2}f_{i}(x^{*}))d_{k}+|d_{k}^{\mathrm T}(B_{k}-\sum\limits_{i=1}^{m}\lambda^{*}_{i}\nabla^{2}f_{i}(x^{*}))d_{k}|)+\epsilon_{k}\|d_{k}\|^{2}.
	\end{align*}
Substituting $\nabla^{2} f_{2}(x_{k})=\begin{pmatrix} 2\ &\ 0\\0\ &\ 2 \end{pmatrix}$ and $\sum\limits_{i=1}^{2}\lambda^{*}_{i}\nabla^{2} f_{i}(x^{*})=\begin{pmatrix}\frac{4}{101}\ &\ 0\\0\ &\ \frac{4}{101} \end{pmatrix}$, and dividing by $\|d_{k}\|^{2}$ on the left and right hand sides of the above inequalities, we obtain
$$2\leqslant\frac{4(2-\sigma)}{100}+(2-\sigma)\frac{|d_{k}^{\mathrm T}(B_{k}-\sum\limits_{i=1}^{m}\lambda^{*}_{i}\nabla^{2}f_{i}(x^{*}))d_{k}|}{\|d_{k}\|^{2}}+\epsilon_{k}.$$
Taking $k\rightarrow\infty$, it follows by (\ref{E5.4}) and $\epsilon_{k}\rightarrow0$ that 
$2\leqslant\frac{4(2-\sigma)}{100},$
this contradicts $\sigma\in(0,1)$.
\end{exam}

From Example \ref{ex1}, notice that different objective functions have a similar amount of descent, then objective functions with different magnitude of Lipschitz-smooth constants will lead to relatively small step size. For this reason, we modify the widely used line search technique to achieve a larger step size.
\begin{algorithm}[H]  
	\caption{\ttfamily aggregated\_Armijo\_line\_search}
	\KwData{$x_{k}\in\mathbb{R}^{n},\lambda^{k}\in\Delta_{m},d_{k}\in\mathbb{R}^{n},\theta_{k}<0,\sigma,\gamma\in(0,1),\alpha=1$}
	\While{$\sum\limits_{i=1}^{m}\lambda_{i}^{k}F_{i}(x_{k}+\alpha d_{k})- \sum\limits_{i=1}^{m}\lambda_{i}^{k}F_{i}(x_{k}) > \sigma\alpha\theta_{k}$}{~\\ $\alpha\leftarrow \gamma\alpha$ }\Return{$\alpha_{k}\leftarrow \alpha$} 
\end{algorithm}

\begin{rem}
	The aggregated Armijo line search is nonmonotone, i.e., it can not guarantee that all objective functions are decreasing in each iteration. It is worth noting that the proposed line search technique is different from the widely used nonmonotone line search techniques \cite{24a,24b}. The line search technique regards all objective functions as a whole, consequently, it can be interpreted as an adaptive weighted sum method from the perspective of scalarization.   
\end{rem}
The variable metric method with BFGS (VMM-BFGS) update is described as follows.
\begin{algorithm}[H] 
	\caption{\ttfamily \ttfamily variable\_metric\_method\_with\_BFGS\_for\_MOPs}
	\KwData{$x_{0}\in\mathbb{R}^{n}$, $\varepsilon\geqslant0$, $B_{0}\leftarrow I$}
	\For{$k=0,1,...$}{ $\lambda^{k}\leftarrow \arg\min\limits_{\lambda\in\Delta_{m}}\frac{1}{2} (\sum\limits_{i=1}^{m}\lambda_{i}\nabla F_{i}(x_{k}))^{\mathrm T}B_{k}^{-1}(\sum\limits_{i=1}^{m}\lambda_{i}\nabla F_{i}(x_{k}))$\\  
		$d_{k}\leftarrow-B_{k}^{-1}(\sum\limits_{i=1}^{m}\lambda_{i}^{k}\nabla F_{i}(x_{k}))$\\
		$\theta_{k}\leftarrow-\frac{1}{2}d_{k}^{\mathrm T}(\sum\limits_{i=1}^{m}\lambda^{k}_{i}\nabla F_{i}(x_{k}))$\\
		\eIf{$|\theta_{k}|\leqslant\varepsilon$}{ ~\\ {\bf{return}} $x_{k}$ }{$\alpha_{k}\leftarrow$ {\ttfamily aggregated\_Armijo\_line\_search}$(x_{k},\lambda^{k},d_{k},\theta_{k})$\\
			$x_{k+1}\leftarrow x_{k}+\alpha_{k}d_{k}$\\
			update $B_{k+1}^{-1}$ as in (\ref{E4.13})}}  
\end{algorithm}
\par Before analyzing the local superlinear convergence of the VMM-BFGS, we present some technical results in the sequel.
\par Firstly, we give error estimates of linear and quadratic approximations.
\begin{lem}
	Suppose that $V\subset\mathbb{R}^{n}$ is a convex set. For any $\epsilon>0$, there exists $\delta>0$ such that  $\|\nabla^{2}F_{i}(x)-\nabla^{2}F_{i}(y)\|<\frac{\epsilon}{2}$ holds for all $i=1,2,...,m$ and $x,y\in V$ with $\|y-x\|<\delta$. Let $\{x_{k}\}$ be a sequence produced by VMM-BFGS. Assume there exists $k_{0}\in\mathbb{N}$ such that 
	$$\|(\sum\limits_{i=1}^{m}\lambda^{k}_{i}\nabla^{2}F_{i}(x_{k})-B_{k})(y-x_{k})\|<\frac{\epsilon}{2}\|y-x_{k}\|,\ \forall k>k_{0},$$
	where $\lambda_{k}$ is the solution to (\ref{DP}) at $x_{k}$.
	Then, for any $x_{k}$, $k>k_{0}$, and any $y\in V$ such that $\|y-x_{k}\|<\delta$, we have
	\begin{equation}\label{E4.14}
		\|\sum\limits_{i=1}^{m}\lambda^{k}_{i}\nabla F_{i}(y)-\sum\limits_{i=1}^{m}\lambda^{k}_{i}\nabla F_{i}(x_{k})-B_{k}(y-x_{k})\|\leqslant\epsilon\|y-x_{k}\|,
	\end{equation}
	and 
	\begin{equation}\label{E4.15}
		|\sum\limits_{i=1}^{m}\lambda^{k}_{i}F_{i}(y) - \sum\limits_{i=1}^{m}\lambda^{k}_{i}F_{i}(x_{k})
		-\sum\limits_{i=1}^{m}\lambda^{k}_{i}\nabla F_{i}(x_{k})^{\mathrm T}(y-x_{k}) -\frac{1}{2}(y-x_{k})^{\mathrm T}B_{k}(y-x_{k})|\leqslant\frac{\epsilon}{2}\|y-x_{k}\|^{2}.
	\end{equation}
\end{lem}
\begin{proof} A direct calculating gives
	\begin{align*}
		&\|\sum\limits_{i=1}^{m}\lambda^{k}_{i}\nabla F_{i}(y)-\sum\limits_{i=1}^{m}\lambda^{k}_{i}\nabla F_{i}(x_{k})-B_{k}(y-x_{k})\|\\
		\leqslant&\|\sum\limits_{i=1}^{m}\lambda^{k}_{i}\nabla F_{i}(y)-\sum\limits_{i=1}^{m}\lambda^{k}_{i}\nabla F_{i}(x_{k})-\sum\limits_{i=1}^{m}\lambda^{k}_{i}\nabla^{2}F_{i}(x_{k})(y-x_{k})\|+\|(\sum\limits_{i=1}^{m}\lambda^{k}_{i}\nabla^{2}F_{i}(x_{k})-B_{k})(y-x_{k})\|
	\end{align*}
	\begin{align*}
		\leqslant&\sum\limits_{i=1}^{m}\lambda^{k}_{i}\|\nabla F_{i}(y)-\nabla F_{i}(x_{k})-\nabla^{2}F_{i}(x_{k})(y-x_{k})\|+\|(\sum\limits_{i=1}^{m}\lambda^{k}_{i}\nabla^{2}F_{i}(x_{k})-B_{k})(y-x_{k})\|\\
		\leqslant&\sum\limits_{i=1}^{m}\lambda^{k}_{i}\frac{\epsilon}{2}\|y-x_{k}\|+\frac{\epsilon}{2}\|y-x_{k}\|\\
		=&\ \epsilon\|y-x_{k}\|.
	\end{align*}
	Next, we prove (\ref{E4.15}).
	\begin{align*}
		&|\sum\limits_{i=1}^{m}\lambda^{k}_{i}F_{i}(y) - \sum\limits_{i=1}^{m}\lambda^{k}_{i}F_{i}(x_{k})
		-\sum\limits_{i=1}^{m}\lambda^{k}_{i}\nabla F_{i}(x_{k})^{\mathrm T}(y-x_{k}) -\frac{1}{2}(y-x_{k})^{\mathrm T}B_{k}(y-x_{k})|\\
		\leqslant&\sum\limits_{i=1}^{m}\lambda^{k}_{i}|F_{i}(y) - F_{i}(x_{k})
		-\nabla F_{i}(x_{k})^{\mathrm T}(y-x_{k}) -\frac{1}{2}(y-x_{k})^{\mathrm T}\nabla^{2} F_{i}(x_{k})(y-x_{k})|\\
		&+|\frac{1}{2}(y-x_{k})^{\mathrm T}(\sum\limits_{i=1}^{m}\lambda^{k}_{i}\nabla^{2}F_{i}(x_{k})-B_{k})(y-x_{k})|\\
		\leqslant&\sum\limits_{i=1}^{m}\lambda^{k}_{i}\frac{\epsilon}{4}\|y-x_{k}\|^{2}+\frac{\epsilon}{4}\|y-x_{k}\|^{2}=\frac{\epsilon}{2}\|y-x_{k}\|^{2}.
	\end{align*}
	The proof is completed.
\end{proof}

We also give some estimate results to mappings $d$ and $\theta$.
\begin{lem}\label{L4.2}
	Let $x\in\mathbb{R}^{n}$ and $a,b\in\mathbb{R}_{+}$ such that $0<a<b$. If
	$$aI\leqslant B(x)\leqslant bI,$$
	then we have
	\begin{itemize}
		\item[$\mathrm{(a)}$] $\frac{a}{2}\|d(x)\|^{2}\leqslant|\theta(x)|\leqslant\frac{b}{2}\|d(x)\|^{2}$,
		\item[$\mathrm{(b)}$] $|\theta(x)|\leqslant\frac{1}{2a}\|\sum\limits_{i=1}^{m}\lambda_{i}\nabla F_{i}(x)\|^{2}$ holds for all $\lambda\in\Delta_{m}$.
	\end{itemize}
\end{lem}
\begin{proof} Inequalities in (a) are obtained directly from (\ref{E3.9}). Next, we prove assertion (b). From (\ref{E3.9}), we have
	$$\theta(x)=-\frac{1}{2} (\sum\limits_{i=1}^{m}\lambda_{i}(x)\nabla F_{i}(x))^{\mathrm T}B(x)^{-1}(\sum\limits_{i=1}^{m}\lambda_{i}(x)\nabla F_{i}(x)),$$
	where $\lambda(x)=(\lambda_{1}(x),\lambda_{2}(x),...,\lambda_{m}(x))$ is the solution of (\ref{DP}). Then
	\begin{align*}
		\theta(x)&=-\frac{1}{2} (\sum\limits_{i=1}^{m}\lambda_{i}(x)\nabla F_{i}(x))^{\mathrm T}B(x)^{-1}(\sum\limits_{i=1}^{m}\lambda_{i}(x)\nabla F_{i}(x))\\
		&\geqslant-\frac{1}{2} (\sum\limits_{i=1}^{m}\lambda_{i}\nabla F_{i}(x))^{\mathrm T}B(x)^{-1}(\sum\limits_{i=1}^{m}\lambda_{i}\nabla F_{i}(x))
	\end{align*}
	holds for all $\lambda\in\Delta_{m}$. It follows by $\theta(x)<0$ that
	$$|\theta(x)|
	\leqslant|\frac{1}{2} (\sum\limits_{i=1}^{m}\lambda_{i}\nabla F_{i}(x))^{\mathrm T}B(x)^{-1}(\sum\limits_{i=1}^{m}\lambda_{i}\nabla F_{i}(x))|\leqslant\frac{1}{2a}\|\sum\limits_{i=1}^{m}\lambda_{i}\nabla F_{i}(x)\|^{2}$$
	holds for all $\lambda\in\Delta_{m}$. 
\end{proof}
\par Next, we give sufficient conditions for local superlinear convergence. The proof of the following results is inspired by \cite[Theorem 5.1]{7}.
\begin{theo}\label{t5.1}
	Denote by $\{x_{k}\}$ a sequence produced by VMM-BFGS. Suppose $V\subset\mathbb{R}^{n}$, $\sigma\in(0,1)$, $k_{0}\in\mathbb{N}$, $a,b,r,\delta,\epsilon>0$ and 
	\begin{itemize}
		\item[$\mathrm{(a)}$] $aI\leqslant B_{k}\leqslant bI,\ \forall k,$
		\item[$\mathrm{(b)}$] $\|\nabla^{2}F_{i}(x)-\nabla^{2}F_{i}(y)\|<\frac{\epsilon}{2},\ \forall x,y\in V$ with $\| x-y\|<\delta$, $i=1,2,...,m,$
		\item[$\mathrm{(c)}$] $\|(\sum\limits_{i=1}^{m}\lambda^{k}_{i}\nabla^{2}F_{i}(x_{k})-B_{k})(y-x_{k})\|<\frac{\epsilon}{2}\|y-x_{k}\|,\ \forall k\geqslant k_{0},\ y\in V,$
		\item[$\mathrm{(d)}$] $\frac{\epsilon}{a}\leqslant1-\sigma$,
		\item[$\mathrm{(e)}$] $\textbf{B}[x_{k_{0}},r]\subset V$,
		\item[$\mathrm{(f)}$] $\| d_{k_{0}}\|\leqslant\min\{\delta,r(1-\frac{\epsilon}
		{a})\}$.
	\end{itemize}
	Then, for all $k\geqslant k_{0}$, we have 
	\begin{itemize}
		\item[$\mathrm{(i)}$] $\| x_{k}-x_{k_{0}}\|\leqslant\frac{1-(\frac{\epsilon}
			{a})^{k-k_{0}}}{1-\frac{\epsilon}
			{a}}\|d_{k_{0}}\|$,
		\item[$\mathrm{(ii)}$] $\| d_{k}\|\leqslant(\frac{\epsilon}
		{a})^{k-k_{0}}\| d_{k_{0}}\|$, 
		\item[$\mathrm{(iii)}$] $\alpha_{k}=1$,
		\item[$\mathrm{(iv)}$] $\| d(x_{k+1})\|\leqslant(\frac{\epsilon}
		{a})\| d_{k}\|$.
	\end{itemize}
	Moreover, the sequence $\{x_{k}\}$ converges to some Pareto critical point $x^{*}\in\mathbb{R}^{n}$ with
	$$\lim\limits_{k\rightarrow\infty}\frac{\|x_{k+1}-x^{*}\|}{\|x_{k}-x^{*}\|}=0,$$
	and the convergence rate of $\{x_{k}\}$ is superlinear.
\end{theo}
\begin{proof} First we prove that if assertions (i) and (ii) hold for some $k\geqslant k_{0}$, then assertions (iii) and (iv) also hold for that $k$.
	\par From assertions (i) and (ii), we have
	\begin{equation}\label{e5.0}
		\|x_{k} + d_{k}-x_{k_{0}}\|\leqslant\|x_{k}-x_{k_{0}}\|+\|d_{k}\|\leqslant\frac{1-(\frac{\epsilon}{a})^{k-k_{0}+1}}{1-\frac{\epsilon}
			{a}}\|d_{k_{0}}\|.
	\end{equation}
	Then, by assumptions (d) and (f), we have
	\begin{equation}\label{e5.1}
		x_{k},x_{k}+d_{k}\in \textbf{B}[x_{k_{0}},r]\ \mathrm{and} \ \|x_{k} + d_{k}-x_{k}\|\leqslant\delta.
	\end{equation}
	Combining (\ref{E4.15}) with assumptions (b) and (c), it follows that
	\begin{align*}
		\sum\limits_{i=1}^{m}\lambda^{k}_{i}F_{i}(x_{k}+d_{k}) - \sum\limits_{i=1}^{m}\lambda^{k}_{i}F_{i}(x_{k})
		&\leqslant\sum\limits_{i=1}^{m}\lambda^{k}_{i}\nabla F_{i}(x_{k})^{\mathrm T}d_{k} +\frac{1}{2}d_{k}^{\mathrm T}B_{k}d_{k} + \frac{\epsilon}{2}\|d_{k}\|^{2}\\
		&=\theta_{k} + \frac{\epsilon}{2}\|d_{k}\|^{2}\\
		&=\sigma\theta_{k} + (1-\sigma)\theta_{k}+\frac{\epsilon}{2}\|d_{k}\|^{2},
	\end{align*}
	where the first equality comes from (\ref{E3.9}) and (\ref{E3.11}).
	Since $\theta_{k}\leqslant0$, from Lemma \ref{L4.2} and assumptions (a) and (d), we have
	$$(1-\sigma)\theta_{k}+\frac{\epsilon}{2}\|d_{k}\|^{2}\leqslant\frac{-a(1-\sigma)+\epsilon}{2}\|d_{k}\|^{2}\leqslant0,$$
	thus,
	$$\sum\limits_{i=1}^{m}\lambda^{k}_{i}F_{i}(x_{k}+d_{k}) - \sum\limits_{i=1}^{m}\lambda^{k}_{i}F_{i}(x_{k})\leqslant\sigma\theta_{k}.$$
	Therefore, Armijo conditions hold for $\alpha_{k}=1$, i.e., assertion (iii) holds. We set $x_{k+1}=x_{k}+d_{k}$, and use (\ref{e5.1}) to obtain
	$$x_{k},x_{k+1}\in\textbf{B}[x_{k_{0}},r]\ \mathrm{and} \ \|x_{k+1}-x_{k}\|\leqslant\delta.$$
	\par Next, we prove assertion (iv). Define $v_{k+1}=\sum\limits_{i=1}^{m}\lambda^{k}_{i}\nabla F_{i}(x_{k+1})$, it follows by Lemma \ref{L4.2} (a) and (b) that
	\begin{equation}\label{e5.2}
		\frac{a}{2}\|d_{k+1}\|^{2}\leqslant|\theta_{k+1}|\leqslant\frac{1}{2a}\|v_{k+1}\|^{2}.
	\end{equation}
	In the following, we estimate $\|v_{k+1}\|$. From inequality (\ref{E4.14}) and assumptions (b) and (c), we have
	\begin{align*}
		\|\sum\limits_{i=1}^{m}\lambda^{k}_{i}\nabla F_{i}(x_{k+1})-\sum\limits_{i=1}^{m}\lambda^{k}_{i}\nabla F_{i}(x_{k})-B_{k}d_{k}\|\leqslant\epsilon\|d_{k}\|.
	\end{align*}
	Substituting $d_{k}= - B_{k}^{-1}(\sum\limits_{i=1}^{m}\lambda^{k}_{i}\nabla F_{i}(x_{k}))$ into the last inequality, we have 
	$$\|v_{k+1}\|\leqslant\epsilon\|d_{k}\|,$$
	which together with (\ref{e5.2}) gives
	$$\|d_{k+1}\|\leqslant\frac{\epsilon}{a}\|d_{k}\|.$$
	So assertion (iv) is valid.
	\par Now we prove that assertions (i) and (ii) hold for all $k\geqslant k_{0}$ by induction. For $k=k_{0}$, they hold trivially. Then, we assume that assertions (i) and (ii) hold for some $k>k_{0}$, as we have already shown, assertions (iii) and (iv) also hold for that $k$. It follows by assertion (iii) that $x_{k+1}=x_ {k}+d_{k}$. Combined with (\ref{e5.0}), we conclude that assertion (i) holds for $k+1$. Assertion (ii) for $k+1$ follows directly from assertion (iv) and our inductive hypothesis.
	\par From assertions (ii) and (iii), we have
	$$\sum\limits_{k=k_{0}}^{\infty}\|x_{k+1}-x_{k}\|=\sum\limits_{k=k_{0}}^{\infty}\|d_{k}\|\leqslant\| d_{k_{0}}\|\sum\limits_{k=k_{0}}^{\infty}(\frac{\epsilon}
	{a})^{k-k_{0}}<\infty,$$
	the last inequality is given by assumption (d).
	Then the sequence $\{x_{k}\}$ is a Cauthy sequence, i.e., there exists $x^{*}$ such that $x_{k}\rightarrow x^{*}$. In view of assumption (a), there exists an infinite index set $K$ and a positive definite matrix $B^{*}$ such that
	$$B_{k}\rightarrow B^{*},\ k\in K.$$ This combined with Lemma \ref{l3.1} (c), the facts $x_{k}\rightarrow x^{*}$ and $d_{k}\rightarrow0$ implies 
	$$d(x^{{*}})=0.$$
	Therefore, by Lemma \ref{l3.1} (b), $x^{*}$ is a Pareto critical point.
	\par We will now prove superlinear convergence of $\{x_{k}\}$. Define 
	$$r_{k}=\frac{(\frac{\epsilon}{a})^{k-k_{0}}}{1-\frac{\epsilon}{a}}\|d_{k_{0}}\|\ \mathrm{and}\ \delta_{k}=(\frac{\epsilon}{a})^{k-k_{0}}\|d_{k_{0}}\|.$$
	By virtue of triangle inequality, assertion (i) and assumptions (e) and (f), we have
	$$\textbf{B}[x_{k},r_{k}]\subset \textbf{B}[x_{k_{0}},r]\subset V.$$ 
	Take any $\xi>0$ and define
	$$\hat{\epsilon}=\min\{a\frac{\xi}{1+2\xi},\epsilon\}.$$
	If $k$ is suffciently large, then
	$$\|\nabla^{2}F_{i}(x)-\nabla^{2}F_{i}(y)\|<\frac{\hat{\epsilon}}{2},\ \forall x,y\in \textbf{B}[x_{k},r_{k}]\  \mathrm{with}\ \| x-y\|<\delta_{k}, i=1,2,...,m,$$ 
	and
	$$\|(\sum\limits_{i=1}^{m}\lambda^{k}_{i}\nabla^{2}F_{i}(x_{k})-B_{k})(y-x_{k})\|<\frac{\hat{\epsilon}}{2}\|y-x_{k}\|,\ \forall k\geqslant k_{0},\ y\in \textbf{B}[x_{k},r_{k}]$$
	hold. Therefore, assumptions (a)-(f) are satisfied for $\hat{\epsilon}$, $r_{k}$, $\delta_{k}$ and $x_{k}$. 
	\par From assertion (i), we have
	\begin{align*}
		\|x^{*}-x_{k}\|&=\lim\limits_{j\rightarrow\infty}\|x_{j}-x_{k}\|\\
		&\leqslant\lim\limits_{j\rightarrow\infty}\frac{1-(\frac{\hat{\epsilon}}
			{a})^{j-k}}{1-\frac{\hat{\epsilon}}
			{a}}\|d_{k}\|\\
		&=\frac{1}{1-\frac{\hat{\epsilon}}
			{a}}\|d_{k}\|.
	\end{align*}
	The latter, together with assertion (iv), leads to
	\begin{equation}\label{E4.19}
		\|x^{*}-x_{k+1}\|\leqslant\frac{1}{1-\frac{\hat{\epsilon}}
			{a}}\|d_{k+1}\|\leqslant\frac{\frac{\hat{\epsilon}}
			{a}}{1-\frac{\hat{\epsilon}}
			{a}}\|d_{k}\|.
	\end{equation}
	Hence,
	\begin{align*}
		\|x^{*}-x_{k}\|&\geqslant\|x_{k+1}-x_{k}\|-\|x^{*}-x_{k+1}\|\\
		&\geqslant\|d_{k}\|-\frac{\frac{\hat{\epsilon}}
			{a}}{1-\frac{\hat{\epsilon}}
			{a}}\|d_{k}\|\\
		&=\frac{1-2\frac{\hat{\epsilon}}
			{a}}{1-\frac{\hat{\epsilon}}
			{a}}\|d_{k}\|.
	\end{align*}
	By the definition of $\hat{\epsilon}$, it is obvious that $1-\frac{\hat{\epsilon}}
	{a}>0$ and $1-2\frac{\hat{\epsilon}}
	{a}>0$. So,
	$$\|x^{*}-x_{k+1}\|\leqslant\frac{\frac{\hat{\epsilon}}
		{a}}{1-2\frac{\hat{\epsilon}}
		{a}}\|x^{*}-x_{k}\|.$$
	Recall the definition of $\hat{\epsilon}$, we get
	$$\|x^{*}-x_{k+1}\|\leqslant\xi\|x^{*}-x_{k}\|.$$
	Since $\xi>0$ is arbitrary chosen, which implies that $\{x_{k}\}$ converges superlinearly to $x^{*}$. 
\end{proof}
As the Theorem \ref{t5.1} showed, when the initial point is sufficiently close to a Pareto critical point, then the whole sequence produced by VMM-BFGS remains in a vicinity of the initial point, and converges superlinearly to some Pareto critical point. 

\section{Numerical Results}
One of the reasons why we propose the variable metric method for MOPs is that comparing with direction-finding subproblems in \cite{7,8a}, (\ref{DP}) in VMM is easy to solve because of unit simplex constraint. Some previous works \cite{28,29} already focused on it. In the sequel, we solve (\ref{DP}) based on Frank-Wolfe/condition gradient method \cite{27}.  
\par Next, we present some
numerical results and demonstrate the numerical performance of VMM-BFGS for different problems. Some comparisons with quasi-Newton method (QNM) \cite{8a} are presented to show the efficiency of our algorithm. All numerical experiments were implemented in Python 3.7 and executed on a personal computer equipped with Intel Core i5-6300U, 2.40 GHz processor, and 4 GB of RAM.
\par In line search, we set $\sigma=0.1$, $\gamma=0.5$. For all test cases solved by VMM-BFGS and QNM, we use $|\theta_{k}|<10^{-8}$ for stopping criterion. The maximum number of iterations is set to 500. The tested algorithms are executed on several test problems, and problem illustration and numerical results are presented in Table 1 and Table 2, respectively. The dimension of variables is presented in the second column (see Table 1). Each problem is computed 200 times with starting points randomly selected in the internals of given lower bounds $x_{L}$ and upper bounds $x_{U}$. As noted in \cite{7}, box constraints can be handled by augmented Armijo line search, which restricts $x_{L}-x\leqslant d\leqslant x_{U}-x$, but $x_{L}=-\infty$ and $x_{U}=+\infty$ are used in our implementation. The subproblem (\ref{DP}) in VMM-BFGS is solved by our codes based on the Frank-Wolfe method, the QCPs in QNM are solved by CVXPY \cite{26}, which is a Python-embedded modelling language for convex optimization problems. The averages of 200 runs record the number of iterations, number of function evaluations and CUP time. 
\begin{table}[H]\centering
	\begin{tabular}{ccccc}
		\hline
		Problem  & n   & $x_{L}$                                     & $x_{U}$                & Reference \\ \hline
		Deb   & 2   & {[}0.1,0.1{]}                           & {[}1,1{]}        & \cite{16}       \\
		JOS1a & 100  & -{[}2,...,2{]}                        & {[}2,...,2{]}    & \cite{18}       \\
		JOS1b & 200 & -{[}2,...,2{]}                        & {[}2,...,2{]}    & \cite{18}       \\
		JOS1c & 500 & -{[}2,...,2{]}                        & {[}2,...,2{]}    & \cite{18}       \\
		JOS1d & 1000 & -{[}2,...,2{]}                        & {[}2,...,2{]}    & \cite{18}       \\
		JOS1e & 100 & -10{[}1,...,1{]}                      & 10{[}1,...,1{]}  & \cite{18}       \\
		JOS1f & 100 & -50{[}1,...,1{]}                      & 50{[}1,...,1{]}  & \cite{18}       \\
		JOS1g & 100 & \multicolumn{1}{l}{-100{[}1,...,1{]}} & 100{[}1,...,1{]} & \cite{18}       \\
		JOS1h & 200 & -100{[}1,...,1{]}                     & 100{[}1,...,1{]} & \cite{18}       \\
		PNR   & 2   & -{[}2,2{]}                            & {[}2,2{]}        & \cite{19}       \\
		WIT0  & 2   & -{[}2,2{]}                            & {[}2,2{]}        & \cite{20}       \\
		WIT1  & 2   & -{[}2,2{]}                            & {[}2,2{]}        & \cite{20}       \\
		WIT2  & 2   & -{[}2,2{]}                            & {[}2,2{]}        & \cite{20}       \\
		WIT3  & 2   & -{[}2,2{]}                            & {[}2,2{]}        & \cite{20}       \\
		WIT4  & 2   & -{[}2,2{]}                            & {[}2,2{]}        & \cite{20}       \\
		WIT5  & 2   & -{[}2,2{]}                            & {[}2,2{]}        & \cite{20}       \\
		WIT6  & 2   & -{[}2,2{]}                            & {[}2,2{]}        & \cite{20}       \\ \hline
	\end{tabular}
	\caption{Description of all test problems used in numerical experiments.}
\end{table}
\begin{exam}
	Consider the following multiobjective optimization:
	$$(\mathrm{Deb})\ \min\limits_{x_{1}>0}(x_{1},\frac{g(x_{2})}{x_{1}})$$
	where $g(x_{2})=2-exp\{-(\frac{x_{2}-0.2}{0.004})^{2}\}-0.8exp\{-(\frac{x_{2}-0.6}{0.4})^{2}\}$.
\end{exam}
Problem Deb is a nonconvex multiobjective optimization problem with a bimodal function $g(x_{2})$, in view of Fig. 1,  the basin of global minimizer $x_2 = 0.2$ is sharp (see Fig.1 (a)), then most of points obtained by VMM-BFGS are local Pareto points (see Fig.1 (b)). 
\begin{figure}[H]
	\centering
	\centering
	\subfigure[]{
		\includegraphics[scale=0.4]{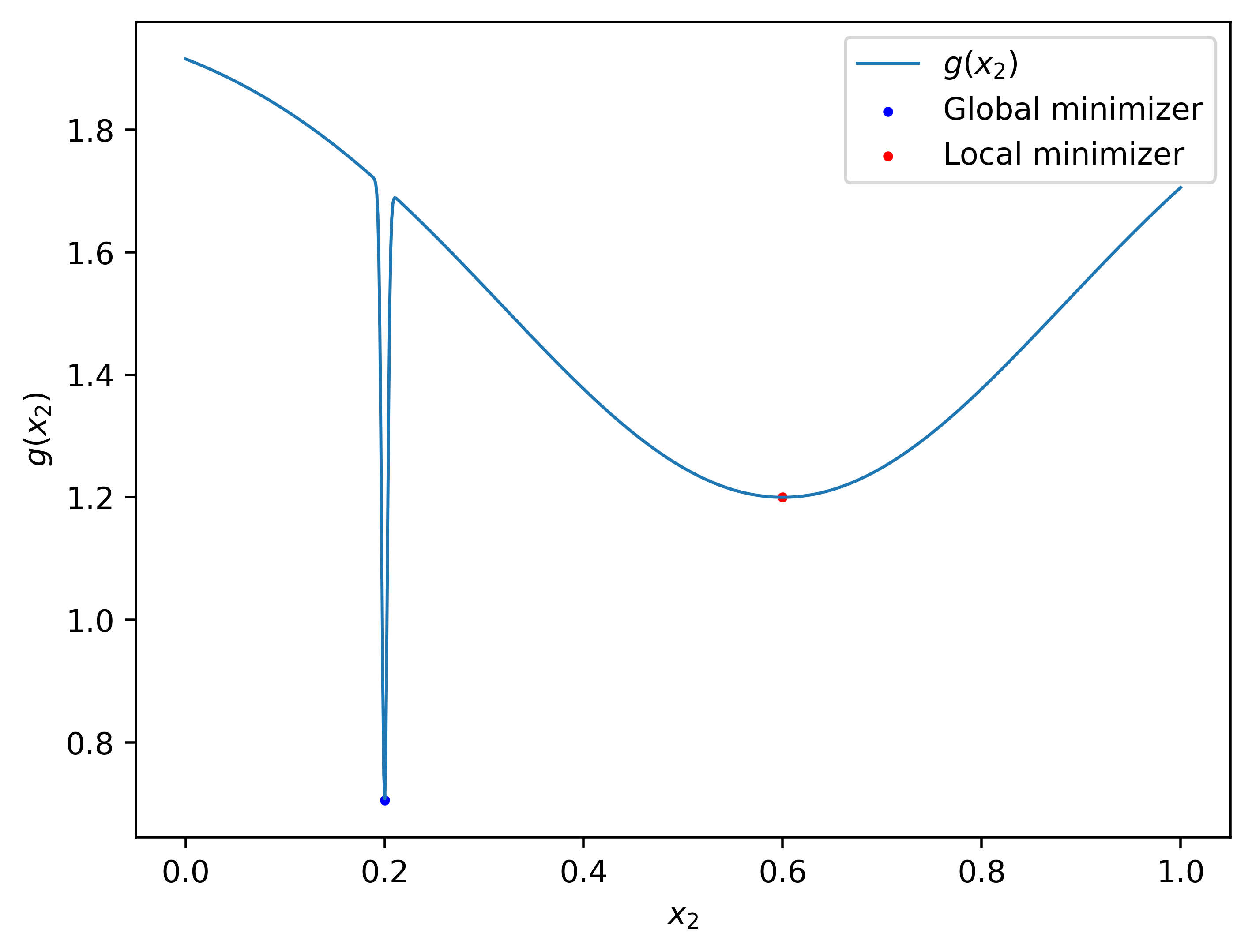}
	}
	\subfigure[]{
		\includegraphics[scale=0.4]{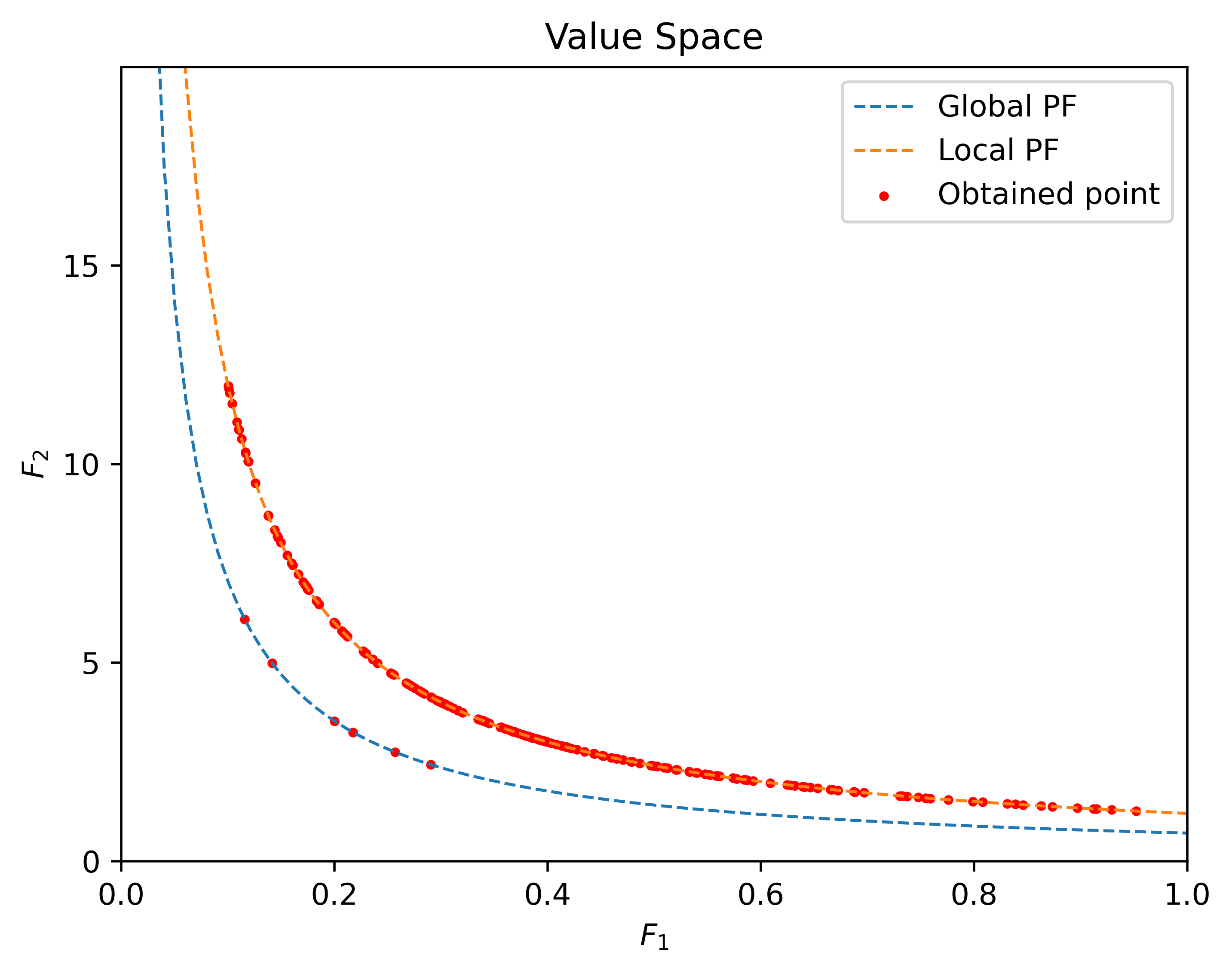} 
	}
	\caption{(a) Local and global minimizers of $g(x_{2})$. (b) Numerical results in value space obtained by VMM-BFGS for problem Deb with 200 random starting points.}	
\end{figure}
\begin{exam}
	Consider the following multiobjective optimization:
	$$(\mathrm{JOS1})\ \min\limits_{x\in \mathbb{R}^{n}}(\frac{1}{n}\sum\limits_{i=1}^{n}x_{i}^{2},\frac{1}{n}\sum\limits_{i=1}^{n}(x_{i}-2)^{2}).$$
\end{exam}
Problem JOS1a-h are simple convex multiobjective problems, the problems pose no challeges to our method and QNM. As it can be seen from Fig. 2, due to the different direction-finding subproblems, VMM-BFGS shows significant improvement in CPU time.
\begin{figure}[H]
	\centering
	\includegraphics[width=6cm,height=4.5cm]{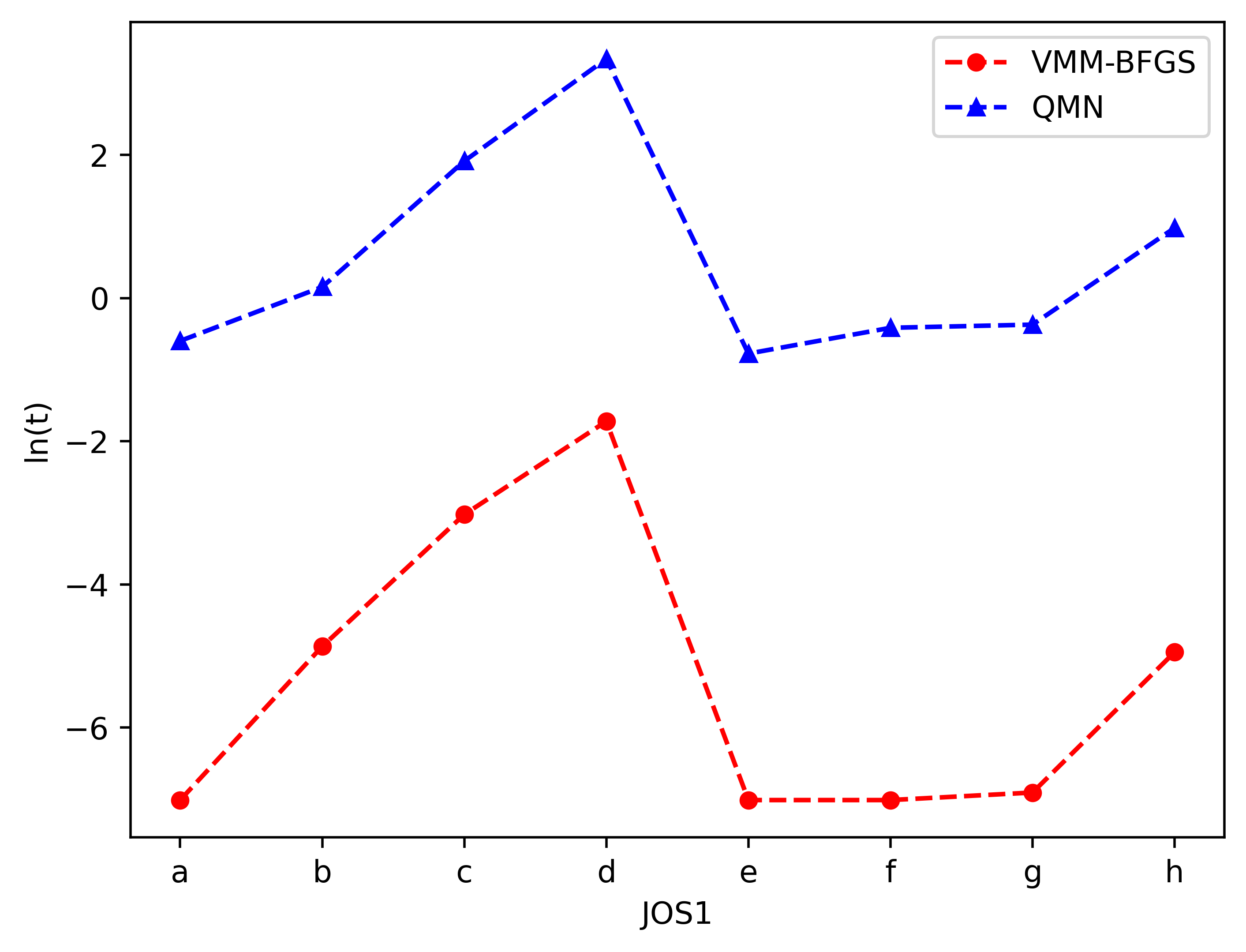}
	\caption{CUP time(t) of VMM-BFGS and QNM for problems JOS1a-h.}	
\end{figure}
\begin{exam}
	Consider the following multiobjective optimization:
	$$(\mathrm{PNR})\ \min\limits_{x\in \mathbb{R}^{2}}(f_{1}(x),f_{2}(x))$$
	where $f_{1}(x)=x_{1}^{4}+x_{2}^{4}-x_{1}^{2}+x_{2}^{2}-10x_{1}x_{2}+0.25x_{1}+20$, and $f_{2}(x)=(x_{1}-1)^{2}+x_{2}^{2}.$
\end{exam}
\par In Fig. 3, PNR with 200 initial points is solved by VMM-BFGS and QNM, respectively. A significant distinction between Fig. 3 (a) and Fig. 3 (b) is that objective functions are not monotone in VMM-BFGS.
\begin{figure}[H]
	\centering
	\subfigure[VMM-BFGS]{
		\includegraphics[scale=0.4]{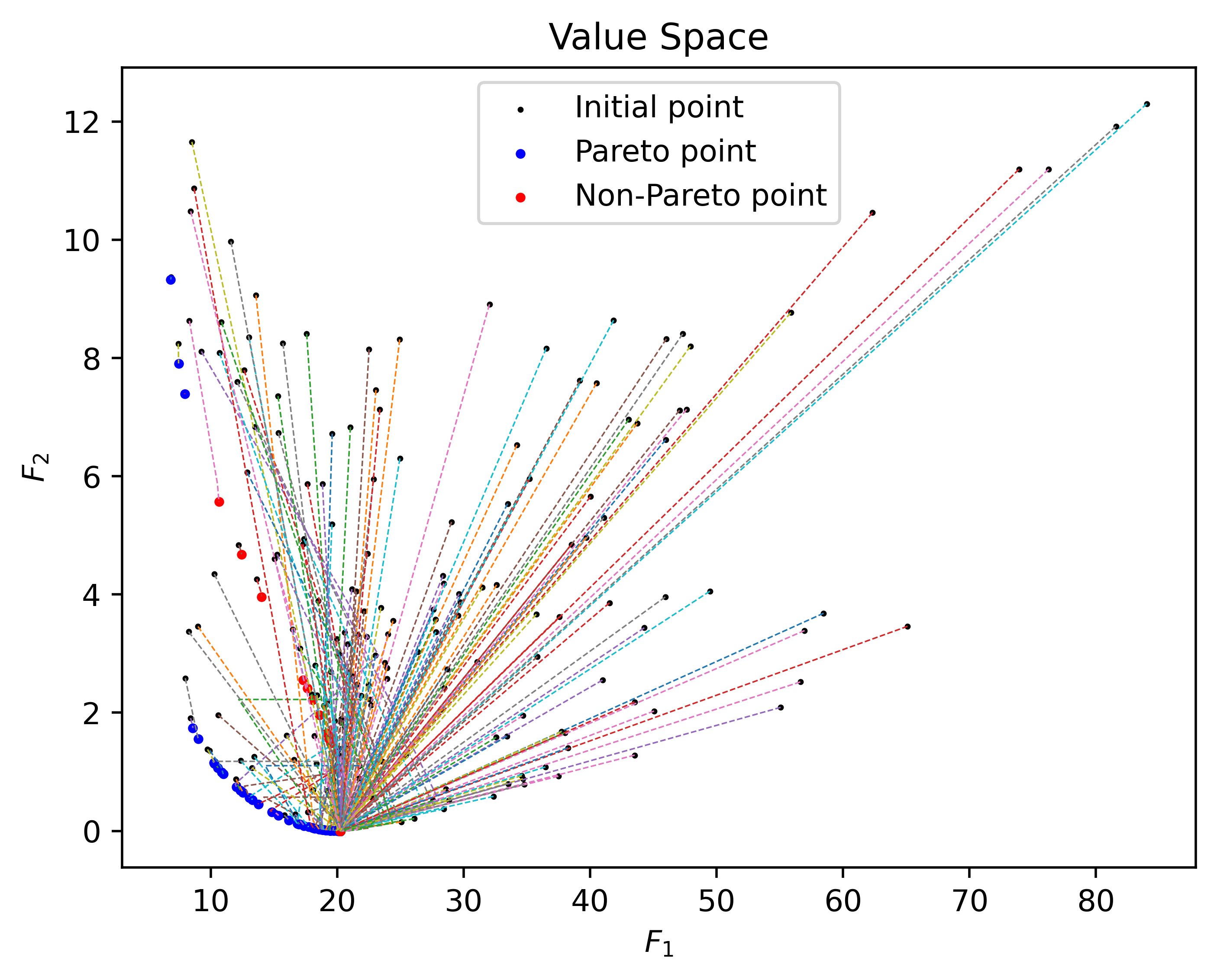}
	}
	\subfigure[QNM]{
		\includegraphics[scale=0.4]{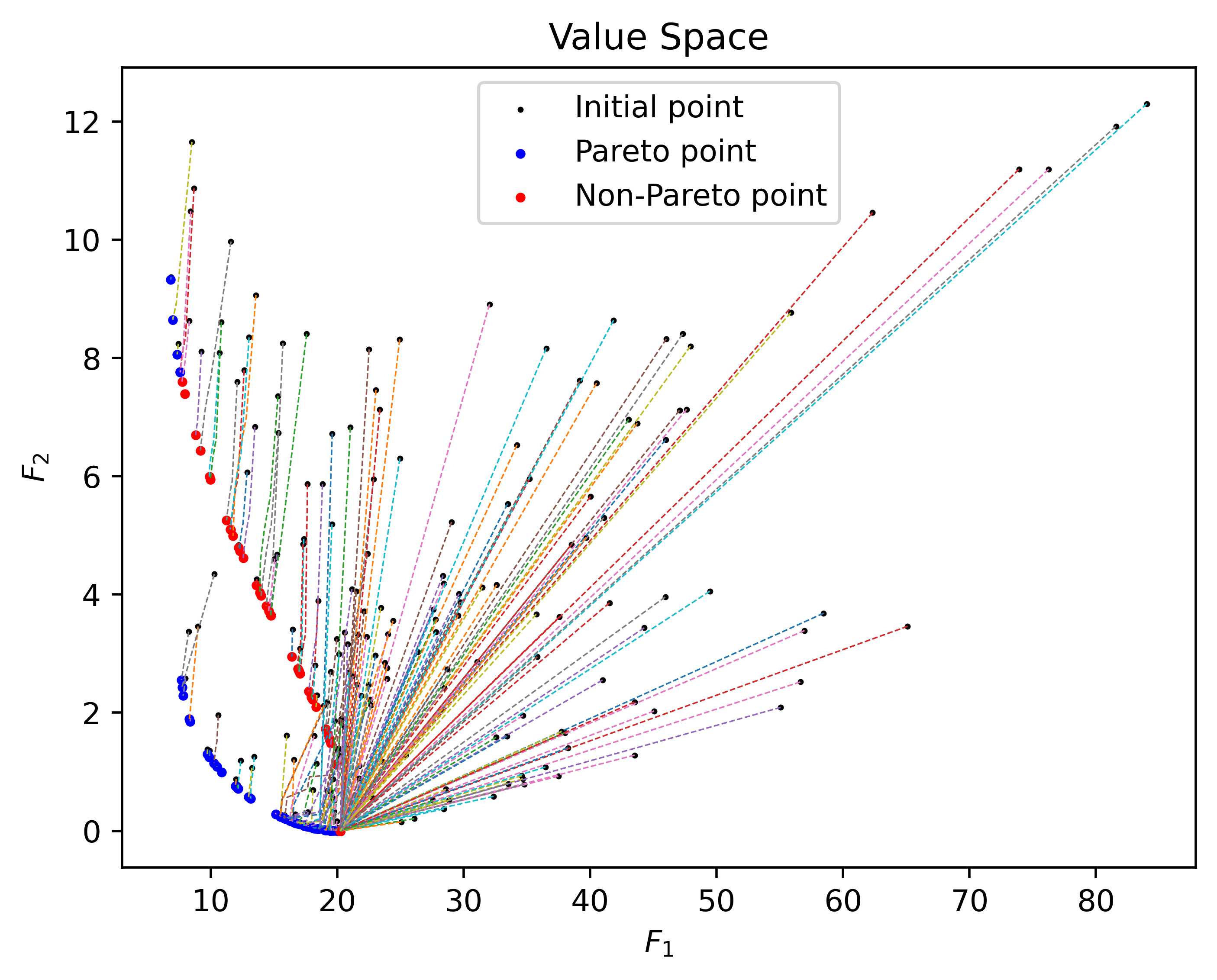} 
	}
	\caption{Iterations in value space obtained by VMM-BFGS and QNM for problem PNR with 200 random starting points.}
\end{figure}
\begin{exam}
	Consider the following multiobjective optimization:
	$$(\mathrm{WIT0})\ \min\limits_{x\in \mathbb{R}^{2}}(f_{1}(x),f_{2}(x))$$
	where $f_{1}(x)=\frac{1}{2}(\sqrt{1+(x_{1}+x_{2})^{2}}+\sqrt{1+(x_{1}-x_{2})^{2}}+x_{1}-x_{2})+0.6e^{-(x_{1}-x_{2})^{2}}$, and $f_{2}(x)=\frac{1}{2}(\sqrt{1+(x_{1}+x_{2})^{2}}+\sqrt{1+(x_{1}-x_{2})^{2}}-x_{1}+x_{2})+0.6e^{-(x_{1}-x_{2})^{2}}.$
\end{exam}
\par While the line search in VMM-BFGS can be  interpreted as an adaptive weighted sum method, i.e., the weights change in each iteration. In Fig. 4, WIT0 has a nonconvex Pareto front, and VMM-BFGS can obtain the solutions to the nonconvex part of the Pareto front. However, the standard weighted sum method can not achieve the solutions.
\begin{figure}[H]
	\centering
	\subfigure[VMM-BFGS]{
		\includegraphics[scale=0.4]{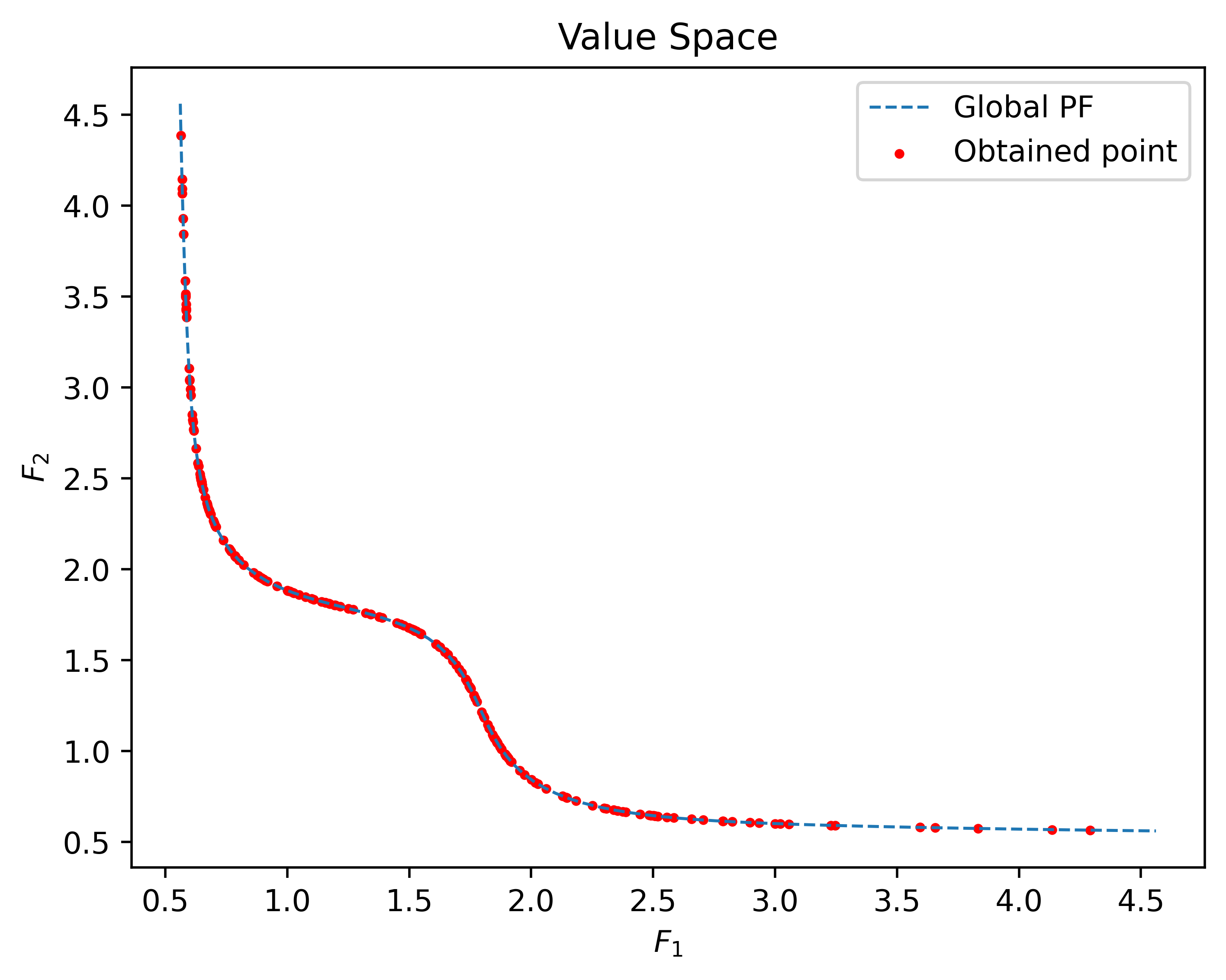}
	}
	\subfigure[weighted sum method]{
		\includegraphics[scale=0.4]{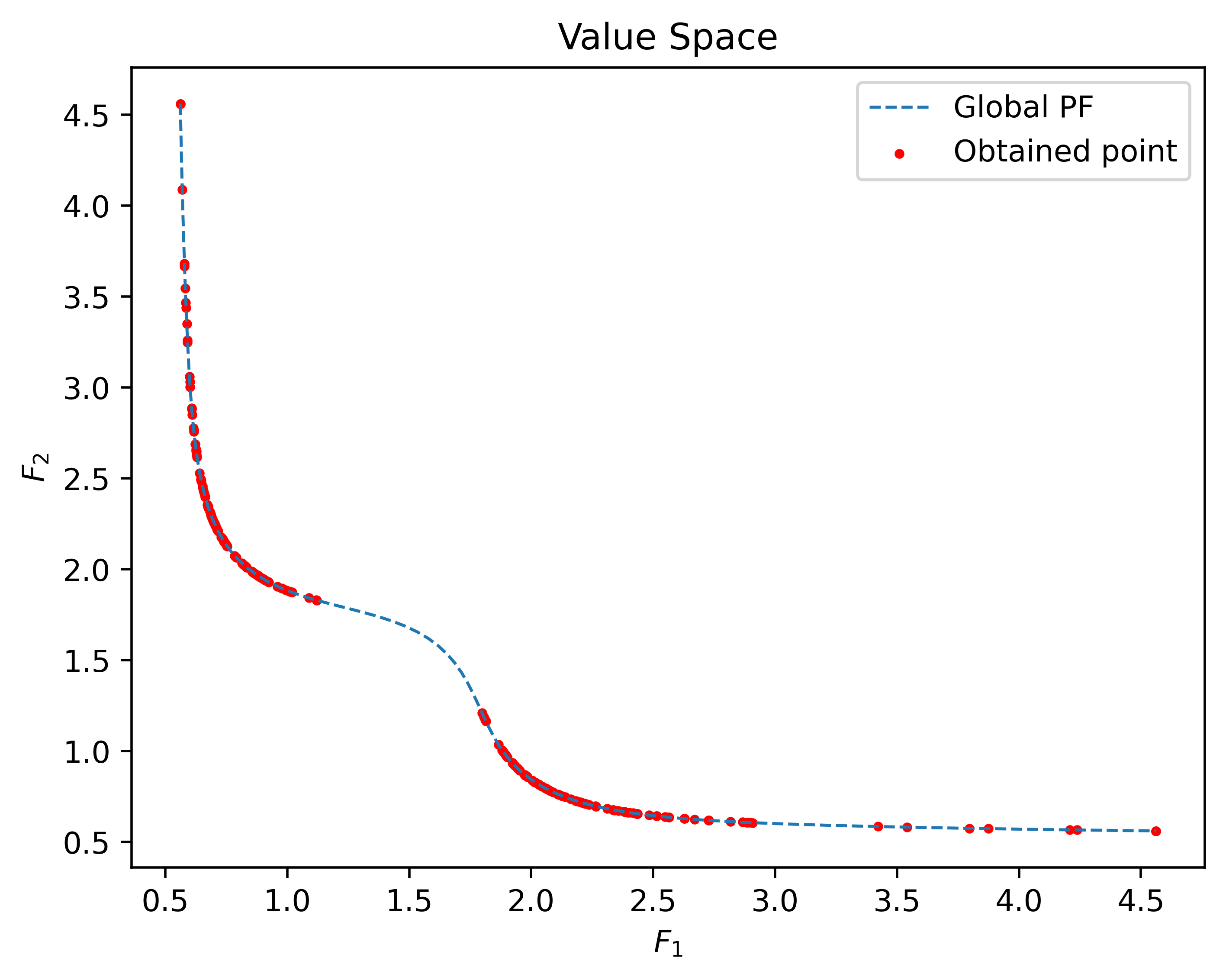} 
	}
	\caption{Numerical results in variable space obtained by VMM-BFGS and weighted sum method for problem WIT0 with 200 random starting points.}
\end{figure}
\begin{exam}
	Consider the following parametric multiobjective optimization:
	$$(\mathrm{WIT})\ \min\limits_{x\in \mathbb{R}^{2}}(f_{1}(x,\lambda),f_{2}(x,\lambda))$$
	where $f_{1}(x,\lambda)=\lambda((x_{1}-2)^{2}+(x_{2}-2)^{2})+(1-\lambda)((x_{1}-2)^{4}+(x_{2}-2)^{8})$, and $f_{2}(x,\lambda)=(x_{1}+2\lambda)^{2}+(x_{2}+2\lambda)^{2}$. Where $\lambda=0,0.5,0.9,0.99,0.999,1$ represent WIT1-6, respectively.
\end{exam}
Problems WIT1-5 are nonconvex MOPs. Comparing with points obtained by VMM-BFGS in Fig. 5 and real Pareto sets in \cite[Fig. 4.1]{20}, we conclude that VMM-BFGS can obtain Pareto solutions to nonconvex multiobjective optimization problems.
\begin{figure}[H]
	\centering
	\subfigure[variable space]{
		\includegraphics[scale=0.4]{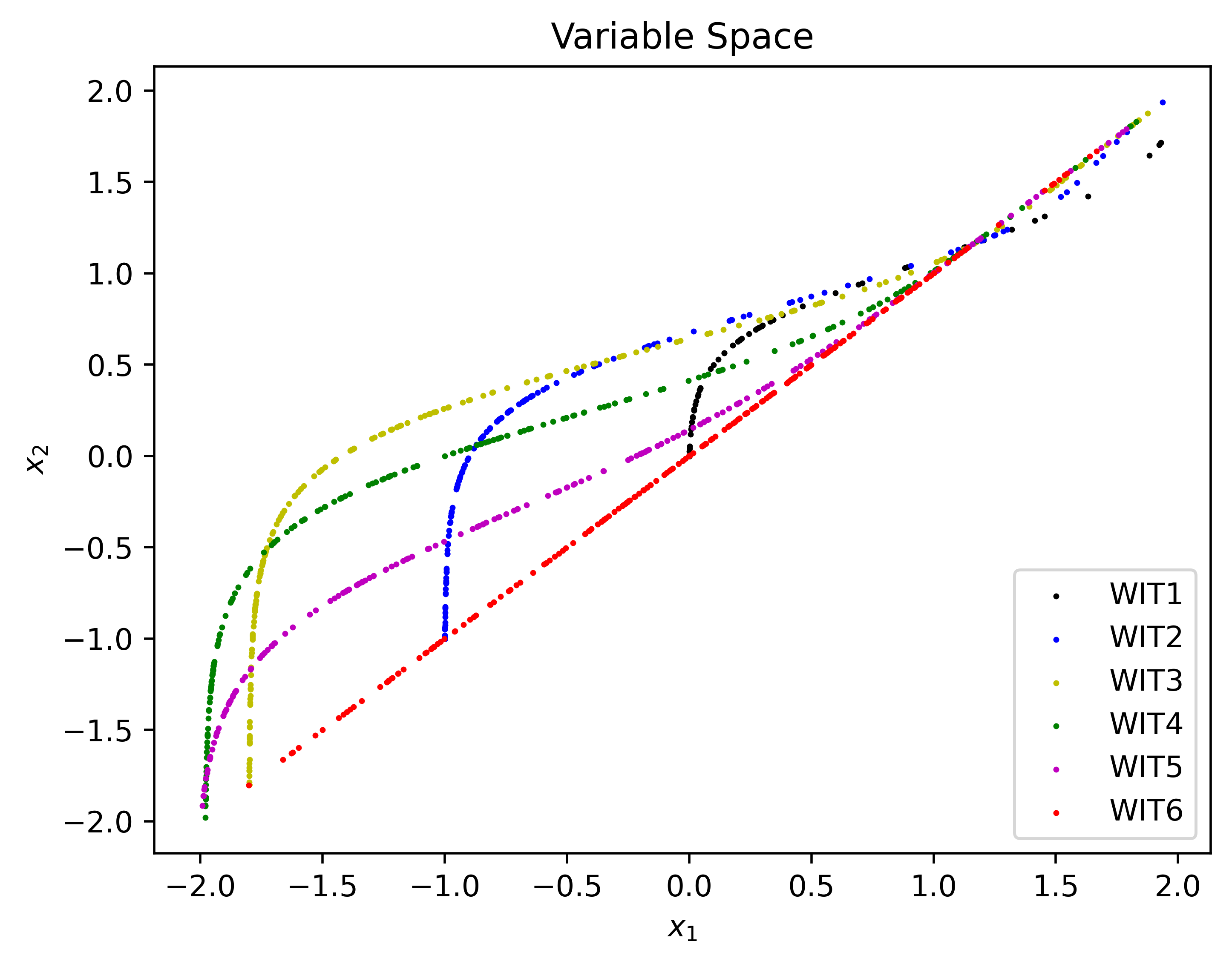}
	}
	\subfigure[objective space]{
		\includegraphics[scale=0.4]{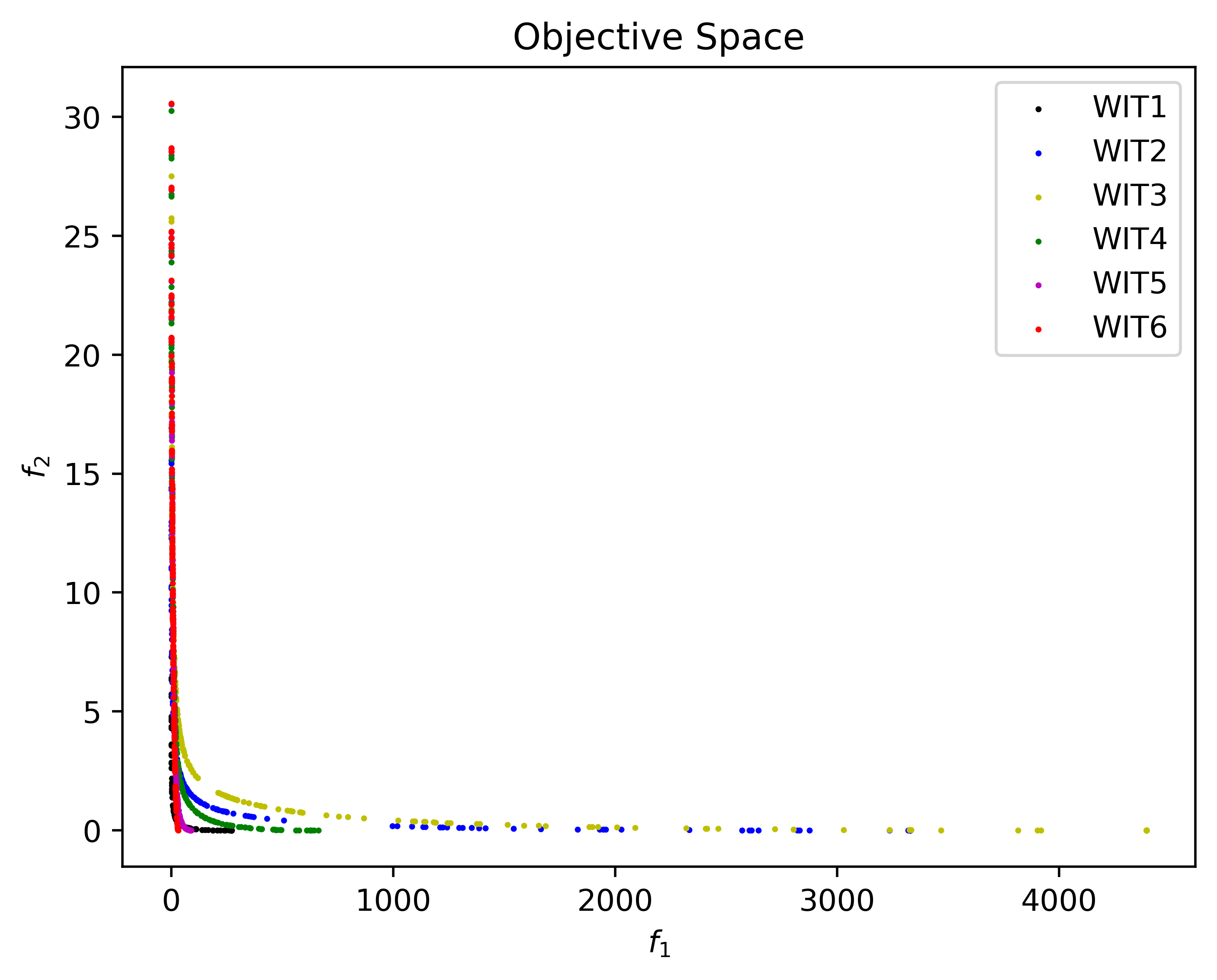} 
	}
	\caption{Numerical results in variable space and objective space obtained by VMM-BFGS for problems WIT1-6 with 200 random starting points.}	
\end{figure}

\begin{table}[H]\centering
	\begin{tabular}{ccccccccc}
		\hline
		Problem &  & \multicolumn{3}{c}{VMM-BFGS} &  & \multicolumn{3}{c}{QNM}  \\ \cline{3-5} \cline{7-9} 
		&  & iter    & feval   & time     &  & iter   & feval  & time   \\ \hline
		Deb     &  & 4.45    & 5.34    & 0.001 7   &  & 84.31  & 517.67 & 1.510 0   \\
		JOS1a   &  & 2.00    & 2.00    & 0.000 9   &  & 2.00   & 2.00   & 0.549 3 \\
		JOS1b   &  & 2.00    & 2.00    & 0.007 7   &  & 2.00   & 2.00   & 1.167 4 \\
		JOS1c   &  & 2.00    & 2.00    & 0.048 6   &  & 2.15   & 2.15   & 6.774 2 \\
		JOS1d   &  & 2.00    & 2.00    & 0.178 7   &  & 2.15   & 2.15   & 28.211 1 \\
		JOS1e   &  & 2.00    & 2.00    & 0.000 9   &  & 2.04   & 2.04   & 0.460 2 \\
		JOS1f   &  & 2.00    & 2.00    & 0.000 9   &  & 2.26   & 2.26   & 0.659 7 \\
		JOS1g   &  & 2.00    & 2.00    & 0.001 1   &  & 2.30   & 2.30   & 0.689 5 \\
		JOS1h   &  & 2.00    & 2.00    & 0.007 1   &  & 4.00   & 4.00   & 2.662 6 \\
		PNR     &  & 2.13    & 3.03    & 0.000 6   &  & 9.35   & 12.60  & 0.169 9 \\
		WIT0    &  & 3.94    & 4.39    & 0.001 6   &  & 3.83   & 4.28   & 0.089 5 \\
		WIT1    &  & 1.88    & 3.12    & 0.000 5   &  & 54.00  & 57.00  & 0.926 6 \\
		WIT2    &  & 2.63    & 3.66    & 0.000 8   &  & 101.34 & 104.95 & 1.622 7 \\
		WIT3    &  & 3.18    & 3.97    & 0.000 9   &  & 48.64  & 50.91  & 0.791 6 \\
		WIT4    &  & 3.26    & 3.94    & 0.000 9   &  & 9.84   & 10.98  & 0.174 1 \\
		WIT5    &  & 3.19    & 3.90    & 0.000 9   &  & 7.49   & 8.54   & 0.134 3 \\
		WIT6    &  & 1.00    & 2.00    & 0.000 4   &  & 1.05   & 2.05   & 0.032 3 \\ \hline
	\end{tabular}
	\caption{Number of average iterations (iter), number of average function evaluations (feval) and average CUP time (time) of VMM-BFGS and QNM implemented on different test problems, each set is executed 200 times.}
\end{table}
In view of the numerical results in Table 2. On the one hand, comparing with QNM, VMM-BFGS requires fewer number of iterations and function evaluations for all test problems besides WIT0. In particular, the required number of iterations and function evaluations for problems Deb and WIT1-3 decreases significantly. On the other hand, VMM-BFGS outperforms QNM greatly in terms of the cost of CPU running times.
\section{Conclusions }
In this paper, a variable metric method for unconstrained multiobjective optimization problems has been presented. Global and strong convergence theorems have been established in the generic framework. In particular, we use a common matrix to approximate Hessian matrices of all objective functions. Moreover, a new nonmonotone line search technique has been proposed to achieve the local superlinear convergence rate of the proposed algorithm, which gives a positive answer to the open question in \cite{10a}. In numerical experiments, the Frank-Wolfe method is applied to solve (\ref{DP}). Numerical results demonstrate the efficiency of the proposed method. There are some features of the proposed method:
\begin{itemize}
	\item[1.] VMM-BFGS is local superlinear convergence under some reasonable assumptions, and direction-finding problems can be solved by the Frank-Wolfe method efficiently. 
	\item[2.] Comparing with the quasi-Newton method in \cite{8a}, the local superlinear convergence of VMM-BFGS holds without the assumption that all objective functions are locally strongly convex, thus, MOPs with some linear objective functions can be solved by the proposed algorithm.
	\item[3.] The line search condition is mild in VMM-BFGS, which does not restrict all objective functions are monotone in each iteration, so a larger step size can be accepted.  
\end{itemize}
\par In the future, we will consider the global convergence of the proposed nonmonotone line search technique, meanwhile, test its abilities in other gradient-based methods with solvable dual problems. Due to the efficiency of the proposed method, a hybrid method with MOEA/D \cite{30} for MOPs will be considered as well.


\begin{thebibliography}{99}
\bibitem{1a} Marler, R. T., Arora, J. S.:
Survey of multi-objective optimization methods for engineering.
Struct. Multidisc. Optim. \textbf{26}(6), 369-395 (2004)
\bibitem{1b} Tapia, M., Coello, C.: Applications of multi-objective evolutionary algorithms in economics and finance: A survey. IEEE Congress on Evolutionary Computation. 532-539 (2007)
\bibitem{1c} Handl, J., Kell,  D. B., Knowles, J.: Multiobjective optimization in bioinformatics and computational biology. IEEE/ACM Trans. Comput. Biol. Bioinform. \textbf{4}(2), 279-292 (2007)
\bibitem{1d} Reed, P. M., Hadka, D., Herman, J. D., Kasprzyk, J. R., Kollat, J. B.: Evolutionary multiobjective optimization in water resources: The past, present, and future. Adv. Water Resour. \textbf{51}, 438-456 (2013) 
\bibitem{1} Gass, S., Saaty. T.:
The computational algorithm for the parametric objective function.
Naval. Res. Logs. \textbf{2}(1), 39-45 (1955)
\bibitem{2} Zadeh, L.:
Optimality and non-scalar-valued performance criteria.
IEEE Trans. Automat. Control.
\textbf{8}(1), 59-60 (1963)
\bibitem{3} Haimes, Y., Lasdon, L., Wismer, D.:
On a bicriterion formulation of the problems of integrated system identification and system optimization. IEEE Trans. Syst. Man Cyber. \textbf{1}(3), 296-297 (1971)
\bibitem{4} Charnes, A., Cooper, W., Ferguson, R.: Optimal estimation of executive compensation by linear programming. Mana. Sci.
\textbf{1}(2), 138-151 (1955)
\bibitem{5} Geoffrion, A. M.: Proper efficiency and the theory of vector maximization. J. Math. Anal. Appl.
\textbf{22}(3), 618-630 (1968)
\bibitem{6a} Mukai, H.: Algorithms for multicriterion optimization. IEEE Trans. Automat. Control. \textbf{25}(2), 177-186 (1980)
\bibitem{6} Fliege, J., Svaiter, B. F.:
Steepest descent methods for multicriteria optimization. Math. Meth. Oper. Res. \textbf{51}(3), 479-494 (2000)
\bibitem{7} Fliege, J., Drummond, L. M. G., Svaiter, B. F.:
Newton's method for multiobjective optimization. SIAM J. Optim.
\textbf{20}(2), 602-626 (2009)

\bibitem{8} Drummong, L. M. G., Iusem, A. N.:
A projected gradient method for vector optimization problems. Comput. Optim. Appl., \textbf{28}(1), 5-29 (2004)

\bibitem{8a} Povalej, \v{Z}.: Quasi-Newton's method for multiobjective optimization. J. Comput. Appl. Math. \textbf{255}, 765-777 (2014)

\bibitem{9} Qu, S. J., Goh, M., Chan, F. T. S.: 
Quasi-Newton methods for solving multiobjective optimization. Oper. Res. Lett. \textbf{39}(5), 397-399 (2011)

\bibitem{9a} Lucambio P\'{e}rez, L. R., Prudente, L. F.: Nonlinear conjugate gradient methods for vector optimization. SIAM J. Optim. \textbf{20}(3), 2690-2720 (2018)

\bibitem{9b} Carrizo, G. A., Lotito, P. A., Maciel, M. C.: Trust region globalization strategy for the
nonconvex unconstrained multiobjective optimization problem. Math. Program. \textbf{159}, 339-369 (2016)

\bibitem{9c} Bonnel, H., Iusem, A. N., Svaiter, B. F.: Proximal methods in vector optimization. SIAM J. Optim. \textbf{15}(4), 953-970 (2005)

\bibitem{10} Moudden, M. E., Mouatasim, A. E.: Accelerated diagonal steepest descent method for unconstrained multiobjective optimization. J. Optim. Theory Appl. \textbf{188}(1), 220-242 (2021)

\bibitem{25} Powell, M. J. D.: Variable metric methods for constrained optimization. Mathematical Programming: The State of the Art. 288-311 (1983)

\bibitem{10a} Ansary, M. A. T., Panda, G.: A modified quasi-Newton method for vector optimization problem. Optimization. \textbf{64}(11), 2289-2306 (2014)

\bibitem{21} Broyden, C. G.: The convergence of a class double-rank minimization algorithms. J. Inst. Math. Appl. \textbf{6}, 76-90 (1970)

\bibitem{22} Fletcher, R.: A new approach to variable metric algorithms. Computer
J. \textbf{13}, 317-322 (1970)

\bibitem{23} Goldfarb, D.: A family of variable metric methods derived by variation mean. Math. Comp. \textbf{23}, 23-26 (1970)

\bibitem{24} Shanno, D. F.: Conditioning of quasi-Newton methods for function minimization. Math. Comp. \textbf{24}, 647-656 (1970)

\bibitem{6b} Luc, D. T.: Theory of Vector Optimization, Springer-Verlag, Berlin (1988)

\bibitem{11} Jian, J. B.: Researches on superlinear and quadratically 
convergent algorithms for nonlinearly constrained optimization. Ph.D. Theies, Xi'an Jiaotong University (2000)

\bibitem{qn} Dennis, J. E., Mor\'{e}, J. J.: Quasi-Newton methods, motivation and theory. SIAM Rev. \textbf{19}, 46-89 (1977)

\bibitem{24a} Grippo, L., Lampariellof, L., Lucidi, S.: A nonmonotone line search technique for Newton's method. SIAM J. Numer. Anal. \textbf{23}(4), 707-716 (1986)

\bibitem{24b} Zhang, H. C., Hager, W. W.: A nonmonotone line search technique and its application to unconstrained optimization. SIAM J. Optim. \textbf{14}(4), 1043-1056 (2004)

\bibitem{28} Wolfe, P.: Finding the nearest point in a polytope. Math. Program. \textbf{11}(1), 128-149 (1976)

\bibitem{29} Hohenbalken, B. V.:
A finite algorithm to maximize certain pseudoconcave functions on polytopes. Math. Program. \textbf{9}(1), 189-206 (1975)

\bibitem{27} Frank, M., Wolfe, P.: An algorithm for quadratic programming. Naval Res. Logist. Q. \textbf{3}, 95-110 (1956)

\bibitem{26} Diamond, S., Boyd, S.: CVXPY: A Python-embedded modeling language for convex optimization. J. Mach. Learn. Res. \textbf{17}, 1-5 (2016)

\bibitem{16} Deb. K.: Multi-objective genetic algorithms: Problem difficulties and construction of test problems. Evol. Comput. \textbf{7}(3), 205-230 (1999)

\bibitem{18} Jin, Y., Olhofer, M., Sendhoff, B.: Dynamic weighted aggregation for evolutionary multi-objective optimization: Why does it work and how? In Proceedings of the Genetic and Evolutionary Computation Conference. 1042-1049 (2001)

\bibitem{19} Preuss, M., Naujoks, B., Rudolph, G.: Pareto set and EMOA behavior for simple multimodal multiobjective functions. In: Parallel Problem Solving from Nature-PPSN IX, 513-522 (2006)

\bibitem{20} Witting, K.: Numerical algorithms for the treatment of parametric multiobjective optimization problems and applications. Ph.D. thesis, Universit$\ddot a$tsbibliothek (2012)

\bibitem{30} Zhang, Q., Li, H.: MOEA/D: A multiobjective evolutionary algorithm based on decomposition. IEEE Trans. Evol. Comput. \textbf{11}(6), 712-731 (2007)
\end{thebibliography}
\end{document}